\documentclass[11pt]{amsart}

\usepackage{tikz}
\usetikzlibrary{arrows,decorations.pathmorphing,backgrounds,positioning,fit,matrix}

\usepackage{comment}
\usepackage{color}
\usepackage{amsmath}
\usepackage{amsthm}
\usepackage{amssymb}
\usepackage{graphicx}
\usepackage{enumerate}
\usepackage{amsfonts}
\usepackage{stackrel}
\usepackage{mathrsfs}
\usepackage{parskip}
\usepackage{mathdots}
\usepackage{mathtools}
\mathtoolsset{showonlyrefs}

\usepackage{color}

\numberwithin{equation}{section}
\theoremstyle{plain}
\newtheorem{Proposition}[equation]{Proposition}
\newtheorem{Corollary}[equation]{Corollary}
\newtheorem*{Corollary*}{Corollary}
\newtheorem{Theorem}[equation]{Theorem}
\newtheorem*{Theorem*}{Theorem}
\newtheorem{Lemma}[equation]{Lemma}
\theoremstyle{definition}

\newtheorem{Example}[equation]{Example}
\newtheorem{Remark}[equation]{Remark}

\allowdisplaybreaks

\usepackage{palatino}

\usepackage{marginnote}

\usepackage{enumitem}
\setlist[enumerate]{leftmargin=*}
\setlist[itemize]{leftmargin=*}

\setlist[enumerate,1]{label=(\alph*),font=\upshape}

\setlist[enumerate,2]{label=(\roman*),font=\upshape}

\def\C{\mathbb{C}}
\def\R{\mathbb{R}}
\def\D{\mathbb{D}}
\def\T{\mathbb{T}}

\def\K{\mathcal{K}}
\def\H{\mathcal{H}}
\def\phi{\varphi}

\renewcommand{\ker}{\operatorname{ker}}

\newcommand{\beqa}{\begin{eqnarray*}}
\newcommand{\eeqa}{\end{eqnarray*}}

\renewcommand{\geq}{\geqslant}
\renewcommand{\leq}{\leqslant}
\renewcommand{\subset}{\subseteq}

\title[Multipliers]{Multipliers and equivalence of functions, spaces, and operators}

\author[C\^{a}mara]{M. C. C\^{a}mara}
\address{Center for Mathematical Analysis, Geometry and Dynamical Systems, Department of Mathematics, Instituto Superior Tecnico, Universidade de Lisboa, 1049-001 Lisboa, Portugal}
\email{cristina.camara@tecnico.ulisboa.pt}

\author[Carteiro]{C. Carteiro}
\address{Center for Mathematical Analysis, Geometry and Dynamical Systems, Department of Mathematics, Instituto Superior Tecnico, Universidade de Lisboa, 1049-001 Lisboa, Portugal}
\email{carlos.carteiro@tecnico.ulisboa.pt}

\author[Ross]{W. T. Ross}
	\address{Department of Mathematics and Statistics, University of Richmond, Richmond, VA 23173, USA}
	\email{wross@richmond.edu}
	


\thanks{The work the first author was partially
supported by FCT/Portugal through CAMGSD, IST-ID, projects UIDB/04459/2020 and UIDP/04459/2020. The work of the second author was partially supported by FCT/Portugal through CAMGSD PhD fellowship UI/BD/153700/2022}

\begin{document}

\begin{abstract}
This paper offers a unified approach to determining when two generalized Toeplitz operators on $L^2$ are ``equivalent''.  
This will be done through multipliers between closed subspaces of $L^2$. Our discussion will include Toeplitz operators (and their duals) on the Hardy space, Hankel operators, asymmetric truncated Toeplitz operators, and dual asymmetric truncated Toeplitz operators. Along the way, there will be a discussion of ``equivalence'' of functions and kernels of generalized Toeplitz operators and a generalization of the Brown--Halmos theorem for this class of operators.
\end{abstract}

\maketitle

\section{Introduction}

This paper relates multipliers between subspaces of $L^2$ and equivalence of generalized Toeplitz operators (defined below in \eqref{9aiuerhgjfkdvc}). To give some context to our results, we begin with some comments on what we mean by, and the significance of, the term equivalence of operators. For two complex Hilbert spaces $\H$ and $\K$, we let $\mathscr{B}(\H, \K)$ denote the bounded linear transformations from $\H$ to $\K$. When $\H = \K$, we let $\mathscr{B}(\H) = \mathscr{B}(\H, \H)$. 
There are various ways to identify $A \in \mathscr{B}(\H_1, \H_2)$ with $B \in \mathscr{B}(\K_1, \K_2)$. 
For example, when $\H_{1} = \H_{2} = \H$, $\K_{1}= \K_{2}= \K$, and there is an isometric isomorphism $U: \K \to \H$ satisfying $A = U B U^{*}$, then $A$ and $B$ are {\em unitarily equivalent}. Such operators $A$ and $B$ share the same norm, spectrum, invariant subspace structure, as well as any assumed properties of normality, subnormality, or hyponormality. When the isometric assumption on $U$ is weakened and $U$ is just a bounded invertible operator, then $A$ and $B$ are {\em similar} ($A = U B U^{-1}$) and thus share the same spectral properties and invariant subspace structure. 

The focus of this paper  is the following, even weaker, but very useful identification of $A$ and $B$. Indeed, $A \in \mathscr{B}(\H_1, \H_2)$ and $B \in \mathscr{B}(\K_1, \K_2)$ are {\em equivalent} if there are bounded {\em invertible} operators $E: \K_{2} \to \H_{2}$ and $F: \H_{1} \to \K_{1}$ such that 
\begin{equation}\label{endnfddnfndffff}
A = E B F.
\end{equation}
The above notion of equivalence  has been  previously used to study spectral properties of operators and the solvability of singular integral equations \cite{MR1246812, MR3806717, MR405144, MR0482317, MR881386, MR790315}. For instance, the main motivation behind  Wiener--Hopf factorization, and its varied applications, relies on the fact that it provides a useful equivalence between a Toeplitz operator (on the Hardy space) with a nonvanishing continuous symbol and a Toeplitz operator with a more easily understood monomial symbol.

One can show that equivalent operators are simultaneously Fredholm (with the same Fredholm defect numbers)  and the range, kernel, and inverse of one operator is uniquely determined by the other. For example, a straightforward argument using the singular value decomposition   will show that two $n \times n$ matrices, thought of as linear transformations on $\C^{n}$, are equivalent if and only if they have the same rank (see Proposition \ref{Prandnksd}). 

This paper studies the above notion of equivalence for the class of generalized Toeplitz operators defined as follows. Let $L^{2}$ denote the standard Lebesgue space of the circle and $L^{\infty}$ the essentially bounded functions in $L^2$. For $\phi \in L^{\infty}$, let $M_{\phi}: L^2 \to L^2$ denote the multiplication operator $M_{\phi} f = \phi f$. For two closed subspaces $\H_1$ and $\H_{2}$ of $L^2$, the {\em generalized  Toeplitz operator} from $\H_1$ to $\H_2$ with {\em symbol} $\phi \in L^{\infty}$, is defined by
\begin{equation}\label{9aiuerhgjfkdvc}
T_\phi^{\H_1,\H_2} : \H_1 \to \H_2, \quad T_\phi^{\H_1,\H_2} = P_{\H_2} M_\phi|_{\H_1}, 
\end{equation}
where $P_{\H}$ denotes the orthogonal projection from $L^2$ onto the closed subspace $\H$.
We will use the notation
$\mathscr{T}(\H_{1}, \H_2) := \{T_\phi^{\H_1,\H_2}: \phi \in L^{\infty}\}$ and 
$
\mathscr{T}(\H) := \mathscr{T}(\H, \H).$

These generalized Toeplitz operators are particular cases of {\em general Wiener-Hopf operators} \cite{MR0251573, MR790315}. When $\H_1 = \H_2 = H^{2}$, the standard Hardy space of the disk \cite{Duren, MR0133008},  the set $\mathscr{T}(H^{2})$ is the set of  {\em Toeplitz operators} \cite{MR1071374}. When $\H_1 = \H_2 = (H^2)^{\perp} =  L^2 \ominus H^2$, $\mathscr{T}((H^{2})^{\perp})$ is the set of  {\em dual Toeplitz operators} \cite{MR1885661}. 
When $\H_1 = H^{2}$ and $\H_2 = (H^{2})^{\perp}$, the set $\mathscr{T}(H^{2},  (H^{2})^{\perp})$ becomes the set of  {\em Hankel operators} \cite{MR985586, MR1630646}. For inner functions $\Theta_1, \Theta_2$ and their corresponding model spaces $\K_{\Theta_j} = H^{2} \ominus \Theta_j H^{2}$, $j = 1, 2$, the operators in $\mathscr{T}(\K_{\Theta_1}, \K_{\Theta_2})$ are the {\em asymmetric truncated Toeplitz operators} \cite{MR3708535, MR3634513, MR2363975}.  When $\H_j = (\K_{\Theta_j})^{\perp} = L^2 \ominus \K_{\Theta_j} = \Theta_j H^2 \oplus (H^2)^{\perp}$, $j = 1, 2$, then $\mathscr{T}((\K_{\Theta_1})^{\perp}, (\K_{\Theta_2})^{\perp})$ is the set of  {\em asymmetric dual truncated Toeplitz operators}  \cite{CamaraRoss, MCKB, MR3759573, BHal, Mbbhbbh}. We will discuss the above classes of operators throughout this paper as motivating examples of generalized Toeplitz operators.

Our path towards an equivalence of generalized Toeplitz operators as in \eqref{endnfddnfndffff} will be through the bounded invertible {\em multipliers} between the ambient subspaces $\H$ and $\K$ of $L^2$. To this end, let 
$\mathscr{G}\!L^{\infty}$ 
denote the set of invertible elements of the algebra $L^{\infty}$ (i.e., $a, \frac{1}{a} \in L^{\infty}$). If $a \in \mathscr{G}\!L^{\infty}$ satisfies 
$
\H = a \K,
$
i.e., is a bounded multiplier from $\K$ onto $\H$, then the generalized Toeplitz operator $T_a^{\K, \H}$ as in \eqref{9aiuerhgjfkdvc} becomes the (surjective) multiplication operator
$T_a^{\K, \H} = M_a|_{\mathcal K}.$
Then, if $a_1, a_2\in \mathscr{G}\!L^\infty$ satisfy
$\H_1 = a_1 \K_1$ and $\H_2 = a_2 \K_2,$  Proposition \ref{2.10AA} shows that the operators $E: \mathcal{K}_2 \to \mathcal{H}_2$ and $F: \mathcal{H}_1 \to \mathcal{K}_1$ defined by 
$$E = T_{\overline a_2^{-1}}^{\K_2, \H_2} \quad \mbox{and} \quad F = T_{a_1^{-1}}^{\H_1, \K_1}$$
are invertible. 
Finally, if the functions $\phi, \psi \in L^{\infty}$ are related by the identity 
\begin{equation}
\label{eqn:factorization}
\psi = \bar a_2 \phi\, a_1,
\end{equation}
we will show in Theorem \ref{2.13A} that 
\begin{equation}
\label{eqn:equiv_operators}
T_\phi^{\H_1, \H_2} = E T_\psi^{\K_1, \K_2} F,
\end{equation}
and thus the generalized Toeplitz operators $T_\phi^{\mathcal H_1, \mathcal H_2}$ and $T_\psi^{\mathcal K_1, \mathcal K_2}$ are equivalent in the sense of \eqref{endnfddnfndffff}. Moreover, they are equivalent by means of the operators $E$ and $F$ which, in themselves, are generalized Toeplitz operators.  Although the result above has a converse in the case of the classical Toeplitz operators (and connects to the well-known Brown--Halmos theorem  \cite[Thm.~8]{BH}), this does not extend to generalized Toeplitz operators  since these operators can be the zero operator without the symbol being zero. We attempt to overcome this difficulty and formulate a sort of converse to this result in Proposition \ref{3,8A}.

When $\H_1 = \H_2 = \H$, $\K_1 = \K_2  = \K$, and $a_1 = a_2 = a$, the identity from \eqref{eqn:equiv_operators} becomes 
\begin{equation}\label{kjvusdyf6sdf7sdf76sdf}
T_{\phi}^{\H} = F^{*}  T_{\psi}^{\K} F
\end{equation}
 and it is natural to ask whether $F^{*} = F^{-1}$, making $F = T^{\H, \K}_{a^{-1}}$ a unitary operator, and hence making the generalized Toeplitz operators $T_{\phi}^{\H}$ and $T_{\psi}^{\K}$ unitarily equivalent. We formulate conditions for when this occurs in Proposition \ref{2.16A} and Corollary \ref{2.17A}.

The equality \eqref{eqn:factorization} defines an equivalence relation between the functions $\phi$ and $\psi$. In \S \ref{sectionfive} we will see plenty of examples of equivalent functions amongst the inner functions and unimodular functions. For example, any Frostman shift of an inner function $\Theta$ is equivalent to $\Theta$, and any two finite Blaschke products of the same degree are equivalent. It is also the case that any two Lipschitz continuous functions with the same winding number about the origin are equivalent. 

These equivalences reduce the study of numerous properties (Fredholmness, invertibility, kernel, range) of a given generalized Toeplitz operator to those of another operator of the same type, but with a simpler symbol. This happens, for instance, when studying Toeplitz operators by using a Wiener-Hopf factorization of their symbols.
Alternatively, one can use the equivalence \eqref{eqn:equiv_operators} to reduce the study of an operator between two spaces $\H_1$, $\H_2$ to that of an operator acting on ``simpler'' spaces $\K_1, \K_2$. We will cover this in \S \ref{TTO6} and \S \ref{DTTO07}. 

The skeptical reader might ask why we are exploring a notion of equivalence of generalized Toeplitz operators when there are the readily available and stronger notions of similarity and unitary equivalence.
The issue is that one might not be able to relate generalized Toeplitz operators using similarity or unitary equivalence (see Proposition \ref{2.16A} and Corollary \ref{2.17A}), but it is often possible using this weaker notion of equivalence. Sometimes this is all that is needed to relate the Fredholm properties of the two operators. There is also the obvious reason for the need for an equivalence beyond unitary equivalence and similarity due to the simple fact that the operators $A$ and $B$ in \eqref{endnfddnfndffff} each act between two {\em different} spaces which, in a sense, almost forces the need for the operators $E$ and $F$ to appear the way they do.

There is often an antilinear unitary mapping $C_{\K}$ on $\K$ for which the generalized Toeplitz operator $T_{\psi}^{\K}$ is {\em complex selfadjoint} in that 
$$C_{\K} T_{\psi}^{\K} C_{\K}^{-1} = (T_{\psi}^{\K})^{*}.$$ When the generalized Toeplitz operators $T_{\phi}^{\H}$ and $T_{\psi}^{\K}$ are related by \eqref{kjvusdyf6sdf7sdf76sdf}, we will show in \S \ref{ALunitaryhhh} that we can define an ``equivalent'' antilinear unitary operator $C_{\H}$ such that complex selfadjointness of $T_{\psi}^{\K}$ with respect to $C_{\K}$ implies complex selfadjointness with respect to $C_{\H}$ and vice-versa.  We will explore this phenomenon for certain classes of  truncated Toeplitz operators in \S \ref{ALunitaryhhh}.

\section{Multipliers, projections, and equivalent operators}\label{section2}

Let $L^2 := L^2(m)$, where $m = d \theta/2\pi$ represents the standard normalized Lebesgue measure on $\T$, and $L^{\infty} = L^{\infty}(m)$ denotes the set of essentially bounded functions on the unit circle  $\T$. For two closed subspaces $\H$ and $\K$ of $L^2$,  the function $a \in L^{\infty}$ is a {\em multiplier} from $\K$ to $\H$ if $a \K \subset \H$. In this section we will focus our attention on the case when $a \in \mathscr{G\!}L^{\infty}$, where 
$$\mathscr{G}\!L^{\infty} := \Big\{g \in L^{\infty}: \frac{1}{g} \in L^{\infty}\Big\}$$ 
(the invertible elements of the algebra $L^{\infty}$) and $a \K = \H$. In this case, the multiplication operator $M_{a}: \K \to \H$, defined by $M_{a} f = a f$,  is bounded and  invertible.

An important class of multipliers arises when $\H$ and $\K$ are both equal to $ H^2$, the classical Hardy space 
$$H^2 = \{f \in L^2: \widehat{f}(n) = 0 \mbox{ for all } n < 0\}.$$ In the above, $(\widehat{f}(n))_{n = -\infty}^{\infty}$ is the sequence of Fourier coefficients associated with an $f \in L^2$. By means of radial boundary values on $\T$, we can also regard $H^2$ as the space of analytic functions on the open unit disk $\D$ whose Taylor coefficients belong to the sequence space $\ell^2$.  Standard facts  about $H^2$ (inner-outer factorization \cite[Ch.~2]{Duren}) say that  $\phi H^2 \subset H^2$ if and only if $\phi \in H^{\infty} := H^2 \cap L^{\infty}$ and $\phi H^{2} = H^2$ if and only if $\phi \in \mathscr{G}\!H^{\infty}$, where 
$$\mathscr{G}\!H^{\infty} := \Big\{g \in H^{\infty}: \frac{1}{g} \in H^{\infty}\Big\}$$
(the invertible elements of the algebra $H^{\infty}$).

Another important example of multipliers appears when $\H$ and $\K$ are model spaces. A {\em model space} corresponding to an {\em inner function} $I$ ($I \in H^{\infty}$ and $|I| = 1$ almost everywhere on $\T$) is 
$$\K_{I} := H^2 \cap (I H^2)^{\perp}.$$ See \cite{MR3526203} for the basics of model spaces. 
 Multipliers between model spaces were initially explored by Crofoot \cite{MR1313454} and then later in \cite{MR3720929}. If $a \in \mathscr{G}\!L^{\infty}$ and $a \K_{\theta} = \K_{\alpha}$ ($\alpha$ and $\theta$ are inner), results from \cite{MR1313454} show that $a \in \mathscr{G}\!H^{\infty}$. In fact, there is a complete description of these types of multipliers \cite[Cor.~18]{MR1313454}.

\begin{Proposition}[Crofoot]\label{2.1A}
For $a \in \mathscr{G}\!L^{\infty}$ and  inner functions $\theta$ and $\alpha$, the following are equivalent: 
\begin{enumerate}
\item  $a \K_{\theta} = \K_{\alpha}$; 
\item $a \in \mathscr{G}\!H^{\infty}$ and 
$\alpha \bar{a}/(\theta a)$ is a unimodular constant function on $\T$.
\end{enumerate}
\end{Proposition}

There can be unbounded multipliers from one model space onto another \cite[Thm.~6.9]{MR3720929} but in this paper we will only consider the {\em bounded} ones. The multipliers from one model space to {\em itself}, are just the non-zero constant functions \cite[Prop.~12]{MR1313454}. Finally,  there may not be {\em any} multiplier between two given model spaces (e.g., when $\theta$ and $\alpha$ are finite Blaschke products of different degrees as this will correspond to finite dimensional model spaces of different dimensions \cite[p.~116]{MR3526203}). However, when $a \K_{\theta} = \K_{\alpha}$, the multiplier $a$ is unique up to a nonzero constant factor \cite[Cor.~13]{MR1313454}.
 
Model spaces $\K_{\theta}$ are kernels  of Toeplitz operators, i.e., $\K_{\theta} = \ker T_{\bar{\theta}}$ \cite[p.~108]{MR3526203}. For a $\phi \in L^{\infty}$, recall the {\em Toeplitz operator} $T_{\phi}$ with symbol $\phi$ is the bounded operator on $H^2$ defined by 
\begin{equation}\label{toooeepizttzzzz}
T_{\phi} f = P_{+}(\phi f),
\end{equation}
 where $P_{+}$ denotes the {\em Riesz projection} \cite{MR1071374, MR0361893}. Recall that $P_{+}$ is the orthogonal projection of $L^2$ onto $H^2$ and is given via Fourier series as  
$$P_{+}\Big(\sum_{n = -\infty}^{\infty} \widehat{f}(n) \xi^n\Big) = \sum_{n = 0}^{\infty} \widehat{f}(n) \xi^n, \quad f \in L^2.$$
 This allows us to formulate an equivalent version of Proposition \ref{2.1A} which will be useful in a moment. 
 
 \begin{Proposition}\label{2.2A}
 For $a \in \mathscr{G}\!L^{\infty}$ and inner functions $\theta$ and $\alpha$, the following are equivalent: 
 \begin{enumerate}
 \item $a \K_{\theta} = \K_{\alpha}$; 
 \item $a \in \mathscr{G}\!H^{\infty}$ and $a = \bar{\theta} \alpha \bar{h}$ for some $h \in \mathscr{G}\!H^{\infty}$.
 \end{enumerate}
 \end{Proposition}
 
 \begin{proof}
 If $a \in \mathscr{G}\!L^{\infty}$ and $a \K_{\theta} = \K_{\alpha}$, then $a \in \mathscr{G}\!H^{\infty}$. Using \cite[Prop.~2.16]{MR3806717},
 \begin{align*}
 \K_{\alpha} = a \K_{\theta} & \iff \ker T_{\bar{\alpha}} = a \ker T_{\bar{\theta}}\\
 & \iff \ker T_{\bar{\alpha}} = \ker T_{a^{-1} \bar{\theta}}\\
 & \iff \bar{\alpha} = a^{-1} \bar{\theta} \bar{h} \quad \mbox{for some $h \in \mathscr{G}\!H^{\infty}$}.
 \end{align*}
 Thus, $a = \bar{\theta} \alpha \bar{h}$ with $h \in \mathscr{G}\!H^{\infty}$. Conversely, if $a \in \mathscr{G}\!H^{\infty}$ and $a = \bar{\theta} \alpha \bar{h}$ with $h \in \mathscr{G}\!H^{\infty}$, then 
$\ker T_{\bar{\alpha}} = \ker T_{\bar{\theta} a^{-1} \bar{h}} = a \ker T_{\bar{\theta}}.$
 Hence, $a \K_{\theta} = \K_{\alpha}$.
 \end{proof}
 
Returning to our general discussion  of multipliers from $\H$ onto $\K$, we have the following useful results for the multipliers between the annihilators $\H^{\perp}$ and $\K^{\perp}$.
 
\begin{Proposition}\label{2.5A}
\begin{enumerate}
\item For $a \in L^{\infty}$, the following are equivalent: 
\begin{enumerate}
\item $a \K \subset \H$;
\item $\bar{a} \H^{\perp} \subset \K^{\perp}$.
\end{enumerate}
\item For $a \in \mathscr{G}\!L^{\infty}$, the following are equivalent: 
\begin{enumerate}
\item $a \K = \H$;
\item $\bar{a} \H^{\perp} = \K^{\perp}$.
\end{enumerate}
\end{enumerate}
\end{Proposition}

\begin{proof}
Suppose $a \K \subset \H$. If $k \in \K$ and $\widetilde{h} \in \H^{\perp}$, then 
$\langle \bar{a} \widetilde{h}, k\rangle = \langle \widetilde{h}, a k\rangle = 0$ and so $\bar{a} \H^{\perp} \subset \K^{\perp}$. The argument above can be reversed. 
\end{proof}

Propositions \ref{2.2A} and \ref{2.5A} yield the following. 

\begin{Proposition}\label{2.6A}
For $b \in \mathscr{G}\!L^{\infty}$ and inner functions $\alpha, \theta$, the following are equivalent: 
\begin{enumerate}
\item $b \K_{\alpha}^{\perp} = \K_{\theta}^{\perp}$; 
\item $b \in \mathscr{G}\!\overline{H^{\infty}}$ and $\alpha b/(\theta \bar{b})$ is a unimodular constant function on $\T$; \item $b \in \mathscr{G}\!\overline{H^{\infty}}$ and $b = \theta \bar{\alpha} h$ for some $h \in \mathscr{G}\!H^{\infty}$.
\end{enumerate}
\end{Proposition}




In what follows, $\H_1, \H_2, \K_1, \K_2$ will be closed subspaces of $L^2$ and, for a closed subspace $\H$ of $L^2$, $P_{\H}$ will denote the orthogonal projection of $L^2$ onto $\H$ (and of course $P_{\H^{\perp}} = I_{\operatorname{d}} - P_{\H}$).

\begin{Theorem}\label{2.8A}
Let $\psi, \phi \in L^{\infty}$. If either $\psi \H^{\perp} \subset \K_{2}^{\perp}$ or $\phi \K_1 \subset \H$, then 
$$T_{\psi \phi}^{\K_1, \K_2} = T_{\psi}^{\H, \K_2} T_{\phi}^{\K_1, \H}.$$
\end{Theorem}

\begin{proof}
If $\psi \H^{\perp} \subset \K_{2}^{\perp}$ we have 
\begin{align*}
T_{\psi \phi}^{\K_1, \K_2}  &= P_{\K_2} (\psi \phi) P_{\K_1}|_{\K_1}\\
& = P_{\K_2} \psi (P_{\H} + P_{\H^{\perp}}) \phi P_{\K_{1}}|_{\K_1}\\
&  = P_{\K_2} \psi P_{\H} \phi P_{\K_1}|_{\K_1}\\
&  = T_{\psi}^{\H, \K_2} T_{\phi}^{\K_1, \H}.
\end{align*}
If $\phi \K_1 \subset \H$, then 
$$
T_{\psi \phi}^{\K_1, \K_2}  = P_{\K_2} \psi \phi P_{\K_1}|_{\K_1}
 = P_{\K_2} \psi P_{\H} \phi P_{\K_1}|_{\K_1}
 = T_{\psi}^{\H, \K_2} T_{\phi}^{\K_1, \H}. \qedhere
$$
\end{proof}

\begin{Remark}\label{2.9A}
A  theorem of Brown and Halmos \cite[Thm.~8]{BH} says that for the classical Toeplitz operators on the Hardy space $H^2$ we have, for $\phi, \psi \in L^{\infty}$, 
\begin{equation}\label{2.1}
T_{\phi} T_{\psi} = T_{\phi \psi} \iff \mbox{$\phi \in \overline{H^{\infty}}$ or $\psi \in H^{\infty}$.}
\end{equation}
Since the condition $\phi \in \overline{H^{\infty}}$ is equivalent to $\phi (H^2)^{\perp} \subset (H^2)^{\perp}$ and the condition $\psi \in H^{\infty}$ is equivalent to $\psi H^2 \subset H^2$, the formula in \eqref{2.1} can be interpreted in terms of multipliers as saying that $T_{\phi} T_{\psi} = T_{\phi \psi}$ if and only if either $\phi$ is a multiplier from $(H^2)^{\perp}$ into $(H^2)^{\perp}$ or $\psi$ is a multiplier from $H^2$ into $H^2$. 

For generalized Toeplitz operators the ``only if'' part of the theorem cannot be always be generalized since the symbols $\phi$ and $\psi$ are not unique (see \eqref{Agammazerooooo} and \eqref{Agammazerooooo1} below). Thus, Theorem \ref{2.8A} appears to be the natural generalization of the Brown--Halmos theorem for generalized Toeplitz operators. We will discuss this further in the next section. 
\end{Remark}

As we approach the main theorem of this section, we note that it follows from Theorem \ref{2.8A} that each $a \in \mathscr{G}\!L^{\infty}$ for which $a \K = \H$ determines two closely related bounded invertible operators from $\K$ onto $\H$.

\begin{Proposition}\label{2.10AA}
Let $a \in \mathscr{G}\!L^{\infty}$ with $a \K = \H$. Then the operators 
\begin{equation}\label{2.2}
T_{a}^{\K, \H} = M_{a}|_{\K} \quad \mbox{and} \quad T_{\bar{a}^{-1}}^{\K, \H} = P_{\H} \bar{a}^{-1} P_{\K}|_{\K}
\end{equation}
are invertible with 
\begin{equation}\label{2.3}
(T_{a}^{\K, \H})^{-1} = T_{a^{-1}}^{\H, \K} \quad \mbox{and} \quad (T_{\bar{a}^{-1}}^{\K, \H})^{-1} = T_{\bar{a}}^{\H, \K}
\end{equation}
which are generalized Toeplitz operators. Moreover, 
\begin{equation}\label{2.4}
T_{\bar{a}^{-1}}^{\K, \H} = [(T_{a}^{\K, \H})^{*}]^{-1}.
\end{equation}
\end{Proposition}

\begin{proof}
The invertibility and the two equalities in \eqref{2.3} follow from Theorem \ref{2.8A} and the identities 
\begin{align*}
T_{\bar{a}^{-1}}^{\K, \H} & = (T_{a^{-1}}^{\H, \K})^{*} = [(T_{a}^{\K, \H})^{-1}]^{*} = [(T_{a}^{\K, \H})^{*}]^{-1}.  \qedhere
\end{align*} 
\end{proof}

The next two results relate the orthogonal projections $P_{\H}$ and $P_{\K}$ when $\H$ and $\K$ are related by multipliers. 

\begin{Proposition}\label{2.11A}
If $a \in \mathscr{G}\!L^{\infty}$ with $\H \subset a \K$, then 
\begin{equation}\label{2.5}
P_{\H} f = a P_{\K} a^{-1} P_{\H} f = P_{\H} \bar{a}^{-1} P_{\K} \bar{a}f \quad \mbox{for all $f \in L^2$}.
\end{equation}
\end{Proposition}

\begin{proof}
Since $a^{-1} \H \subset \K$ we have that 
$$P_{\H} f = a a^{-1} P_{\H} f =  a P_{\K} a^{-1} P_{\H} f \quad \mbox{for all $f \in L^2$}.$$
On the other hand, by Proposition \ref{2.5A}, $a^{-1} \H \subset \K \implies \bar{a}^{-1}\K^{\perp} \subset \H^{\perp}$, and so 
\begin{align*}
P_{\H} f & = P_{\H} \bar{a}^{-1} \bar{a} P_{\H} f = P_{\H} \bar{a}^{-1} P_{\K} \bar{a} P_{\H} f  \quad \mbox{for every $f \in L^2$.} \qedhere
\end{align*}
\end{proof}

\begin{Corollary}\label{2.12A}
If $a \in \mathscr{G}\!L^{\infty}$ with $a \K = \H$, then 
\begin{equation}\label{2.6}
P_{\H} = \widetilde{E} (P_{\K} \bar{a} I_{\operatorname{d}}) = (a P_{\K}) \widetilde{F},
\end{equation}
where
$
\widetilde{E} = P_{\H} \bar{a}^{-1} P_{\K}$ and $\widetilde{F} = P_{\K} a^{-1} P_{\H}
$
are such that 
$
T_{\bar{a}^{-1}}^{\K, \H} = \widetilde{E}|_{\K}$ and $T_{a^{-1}}^{\H, \K} = \widetilde{F}|_{\H}.
$
\end{Corollary}

We now formulate our main theorem regarding equivalence of generalized Toeplitz operators. Various reformulations of this will be provided below (for Toeplitz, Hankel, and truncated Toeplitz operators). In \S 6 we will see an application of this to the inverse of a truncated Toeplitz operator, and in \S 7 we reformulate this for dual truncated Toeplitz operators.

\begin{Theorem}\label{2.13A}
Let $\H_1, \H_2, \K_1, \K_2$ be closed subspaces of $L^2$ and  $a_1, a_2 \in \mathscr{G}\!L^{\infty}$ with $a_1 \K_1 = \H_1$ and $a_2 \K_2 = \H_2$. If $\phi, \psi \in \mathscr{G}\!L^{\infty}$ with $\phi = \bar{a}_2 \psi a_1$, then 
\begin{equation}\label{2.9}
T_{\phi}^{\K_1, \K_2} = E T_{\psi}^{\H_1, \H_2} F,
\end{equation}
where 
\begin{equation}\label{2.10}
E = T_{\bar{a}_2}^{\H_2, \K_2} \quad \mbox{and} \quad F =  T_{a_1}^{\K_1, \H_1} = M_{a_1}|_{\K_1}.
\end{equation}
Thus, the generalized Toeplitz operators $T_{\phi}^{\K_1, \K_2}$ and $T_{\psi}^{\H_1, \H_2}$ are equivalent. 
\end{Theorem}

\begin{proof}
The equivalence in \eqref{2.9} follows from Corollary \ref{2.12A} since we have 
\begin{align*} P_{\K_2} & = T_{\bar{a}_2}^{\H_2, \K_2} P_{\H_2} \bar{a}_{2}^{-1} I_{L_2} \quad \mbox{and} \quad P_{\K_1} = a_{1}^{-1} P_{\H_1} a_1 P_{\K_1} = a_{1}^{-1} T_{a_1}^{\K_1, \H_1}. \qedhere
\end{align*}
\end{proof}

\begin{Remark}\label{2.13B}
By Theorem \ref{2.8A}, the identity 
\begin{equation}\label{2.10A}
T_{\phi}^{\K_1, \K_2} = E T_{\psi}^{\H_1, \H_2} F,
\end{equation}
where $\phi = \bar{a}_2 \psi a_1$, $E = T_{\bar{a}_2}^{\H_2, \K_2}$,  and $F = T_{a_1}^{\K_1, \H_1},$
holds if we impose the weaker conditions $a_1 \K_1 \subset \H_1$ and $a_2 \K_2 \subset \H_2$. However, with these weaker conditions, the operators $E$ and $F$ may not be invertible and so \eqref{2.10A} does not provide an equivalence between $T_{\phi}^{\K_1, \K_2}$ and $T_{\psi}^{\H_1, \H_2}$ as in \eqref{endnfddnfndffff}.
\end{Remark}


When $\H_1 = \H_2 = \H$, $\K_1 = \K_2 = \K$, Theorem \ref{2.13A} becomes the following.

\begin{Corollary}\label{2.15A}
Let $a_1, a_2 \in \mathscr{G}\!L^{\infty}$ such that $a_i \K = \H$, $i = 1, 2$. Then if $\phi, \psi \in L^{\infty}$ with 
$\phi = \bar{a}_2\psi a_1$, we have 
\begin{equation}\label{2.12}
T_{\phi}^{\K} = E T_{\psi}^{\H} F,
\end{equation}
with 
$
E = T_{\bar{a}_2}^{\H, \K}$ and $F = T_{a_1}^{\K, \H} = M_{a_1}|_{\K}.$
\end{Corollary}

At this point, it makes sense to ask under what conditions does \eqref{2.12} define a similarity or a unitary equivalence. Regarding similarity, we have the following.

\begin{Proposition}\label{2.16A}
With the assumptions and notation as in Corollary \ref{2.15A}, the following are equivalent:
\begin{enumerate}
\item  the operators $T_{\phi}^{\K}$ and $T_{\psi}^{\H}$ are similar via \eqref{2.12}; 
\item $E = F^{-1}$; 
\item $T_{\bar{a}_2 - a_{1}^{-1}}^{\H, \K} = 0$; 
\item $T_{\bar{a}_2^{-1} - a_1}^{\K, \H} = 0$; 
\item $T_{1 - \bar{a}_2 a_1}^{\K} = 0$; 
\item $T_{1 - \bar{a}_2^{-1} a_{1}^{-1}}^{\H} = 0$.
\end{enumerate}
\end{Proposition}

\begin{proof}
Observe that Proposition \ref{2.10AA} yields 
\begin{align*}
E = F^{-1} & \iff T_{\bar{a}_2}^{\H, \K} = T_{a_{1}^{-1}}^{\H, \K}
 \iff T_{\bar{a}_2 - a_{1}^{-1}}^{\H, \K} = 0\\
& \iff P_{\K} \bar{a}_2 \big(1 - a_{1}^{-1} \bar{a}_2^{-1}\big) P_{\H}|_{\H} = 0\\
& \iff (P_{\K} \bar{a}_2 P_{\H}) \Big(P_{\H} (1 - a_{1}^{-1} \bar{a}_2^{-1}) P_{\H}\Big)|_{\H} = 0\\
& \iff E T_{1 - \bar{a}_2^{-1} a_{1}^{-1}}^{\H} = 0
 \iff T_{1 - \bar{a}_2^{-1} a_{1}^{-1}}^{\H} = 0,
\end{align*}
where we have taken into account that, by Corollary \ref{2.5A}, $\bar{a}_2 \H^{\perp} = \K^{\perp}$ and $E$ is invertible. Analogously, Proposition \ref{2.10AA} yields
\begin{align*}
E^{-1} = F \iff T_{\bar{a}_2^{-1} - a_1}^{\K, \H} = 0
 \iff E^{-1} T_{1 - a_1 \bar{a}_2}^{\K} = 0
 \iff T_{1 - a_{1} \bar{a}_2}^{\K} = 0,
\end{align*}
where we took into account that $\bar{a}_2^{-1} \K^{\perp} = \H^{\perp}$.
\end{proof}

\begin{Remark}
Naturally, the operator $F^{-1} T_{\psi}^{\H} F: \K \to \K$ is similar to $T^{\H}_{\psi}$ for any invertible $F$ on $\K$. However, in general, it is not a generalized Toeplitz operator of the form $T^{\K}_{\psi}$ (see Example \ref{6.3}).
\end{Remark}

Regarding the possibility that \eqref{2.12} defines a unitary equivalence between the operators $T_{\phi}^{\K}$ and $T_{\psi}^{\H}$, which can happen only when the multiplication operator $M_{a_1}: \K \to \H$ is unitary, we have the following. 

\begin{Corollary}\label{2.17A}
With the same assumptions as in Corollary \ref{2.15A}, the following are equivalent: 
\begin{enumerate}
\item the identity in \eqref{2.12} is a unitary equivalence;
\item  $E = F^{-1} = F^{*}$; 
\item $T_{1 - \bar{a}_2 a_1}^{\K} = 0$ and $T_{1 - |a_1|^2}^{\K} = 0$.
\end{enumerate}
\end{Corollary}

\begin{proof}
The condition $T_{1 - \bar{a}_2 a_1}^{\K} = 0$ is equivalent to $E = F^{-1}$ while the condition $T_{1 - |a_1|^2}^{\K} = 0$ is equivalent to $F^{-1} = F^{*}$.
\end{proof}

\begin{Corollary}\label{2.18A}
When $\H_1 = \H_2 = \H$, $\K_1 = \K_2 = \K$, $a_1 = a_2 = a \in \mathscr{G}\!L^{\infty}$, and $a \K = \H$, the identity in \eqref{2.12} takes the form 
\begin{equation}\label{2.16}
T_{\phi}^{\K} = F^{*} T_{\psi}^{\H} F,
\end{equation}
with $F = T_{a}^{\K, \H} = M_{a}|_{\K}$ and $\psi = |a|^{-2} \phi$,
and $T_{\phi}^{\K}$ is unitarily equivalent to $T_{\psi}^{\H}$ if and only if $T_{1 - |a|^2}^{\K} = 0,$
or equivalently, $
T_{\bar{a} - a^{-1}}^{\H, \K} = 0.
$
In this case, $F^{*} = F^{-1}$, making $M_{a}|_{\K}: \K \to \H$ an isometric isomorphism.

\end{Corollary}

The question of similarity or unitary equivalence is therefore related to the delicate question of characterizing the functions $\eta \in L^{\infty}$ for which $T_{\eta}^{\H, \K} = 0$. For certain $\H$ and $\K$, this takes place only when the symbol is zero. For other choices of $\H$ and $\K$, however, there is often a rich variety of symbols $\eta$ corresponding to the zero generalized Toeplitz operator. We will explore this in the next several sections. Let us now discuss Theorem \ref{2.13A} in the case of the classical Toeplitz and Hankel operators as well as truncated Toeplitz operators.

\subsection*{Classical Toeplitz operators}
Consider the classical Toeplitz operators $T_{\phi}$ on $H^2$, where $\phi \in L^{\infty}$ (recall the definition in \eqref{toooeepizttzzzz}). Setting $\H_1 = \H_2 = \K_1 = \K_2 = H^2$ in Theorem \ref{2.13A}, we have the following equivalence.

\begin{Theorem}\label{2.19A}
Let $a_1, a_2 \in \mathscr{G}\!H^{\infty}$ and $\phi, \psi \in L^{\infty}$ with $\phi = \bar{a}_2 \psi a_1$. Then
\begin{equation}\label{2.19}
T_{\phi} = T_{\bar{a}_2} T_{\psi} T_{a_1}
\end{equation}
and thus $T_{\phi}$ is equivalent to $T_{\psi}$. 
\end{Theorem}

The alert reader will recognize an alternative proof of Theorem \ref{2.19A} from the well known Brown--Halmos theorem  \cite[Thm.~8]{BH}. One will also recognize a converse to Theorem \ref{2.19A} in that if $T_{\phi}  = T_{\bar{a}_2} T_{\psi} T_{a_1}$ with $a_1, a_2 \in \mathscr{G}\!H^{\infty}$, then $T_{\phi} = T_{\bar{a}_2 \psi a_1}$ and, by the uniqueness of symbols for Toeplitz operators \cite[p.~179]{MR0361893}, $\phi = \bar{a}_2 \psi a_1$. In the following section we will recast this in terms of equivalence the symbols of Toeplitz operators. The uniqueness of these symbols, and the fact that $H^{\infty} \cap \overline{H^{\infty}} = \C$, also shows that \eqref{2.19} {\em never} provides a similarity between two Toeplitz operators unless we are in the degenerate case where $a_1$ and $a_2$ are constant functions. 

\subsection*{Classical Hankel operators} For $\phi \in L^{\infty}$ recall the classical {\em Hankel operator} $H_{\phi}: H^2 \to (H^2)^{\perp}$ defined by $H_{\phi} = P_{-} \phi P_{+}|_{H^2}$ \cite{MR985586,MR1630646}. Taking $\H_1 = \K_1 = H^2$ and $\H_2 = \K_2 = (H^2)^{\perp}$ in Theorem \ref{2.13A}, we have the following. 

\begin{Theorem}\label{2.20A}
Let $a_1, a_2 \in \mathscr{G}\!H^{\infty}$ and $\phi, \psi \in L^{\infty}$ with $\phi = a_2 \psi  a_1$. Then 
$$H_{\phi} = S_{a_{2}} H_{\psi} T_{a_{1}},$$
where $S_{a_{2}} = P_{-} a_{2} P_{-}|_{(H^2)^{\perp}}$ is an invertible dual Toeplitz operator  and  $T_{a_{1}}$ is an invertible Toeplitz  operator. Thus, $H_{\phi}$ and $H_{\psi}$ are equivalent Hankel operators. 
\end{Theorem}

See \cite{MR4041541, MR1885661} for more on the dual of a Toeplitz operator. The non-uniqueness of symbols of Hankel operators (indeed $H_{\phi} \equiv 0$ if and only of $\phi \in H^{\infty}$) prevents us from formulating a converse of Theorem \ref{2.20A} as we did for Toeplitz operators. See \S \ref{sectionthree} where we address this issue further. 

\subsection*{Truncated Toeplitz operators}

This section gives a version of Theorem \ref{2.13A} in the special case where the generalized Toeplitz operator becomes a truncated Toeplitz operator. 
For an inner function $\alpha$, we let $P_{\alpha}$ denote the orthogonal projection of $L^2$ onto $\K_{\alpha}$. 
For $\phi \in L^{\infty}$ and inner functions $\alpha$ and $\gamma$, let 
\begin{equation}\label{ATTO999}
A_{\phi}^{\gamma, \alpha}: \K_{\gamma} \to \K_{\alpha}, \quad A^{\gamma, \alpha}_{\phi} f := P_{\alpha}(\phi f)
\end{equation}
 be the {\em asymmetric truncated Toeplitz operator} from $\K_{\gamma}$ to $\K_{\alpha}$ with symbol $\phi$  \cite{MR3708535, MR3634513}. When $\gamma = \alpha$ we write $A^{\gamma, \gamma}_{\phi} = A^{\gamma}_{\phi}$, which is called the {\em truncated Toeplitz operator}  on $\K_{\gamma}$ with symbol $\phi$ \cite{MR2363975} .  Under the right circumstances, one can define $A^{\gamma, \alpha}_{\phi}$ when $\phi \in L^2$ (define it densely on bounded functions in $\K_{\gamma}$ and consider if it has a bounded extension to $\K_{\gamma}$). Surprisingly, not every one of these operators can be represented by an $L^{\infty}$ symbol $\phi$ \cite{MR2679022}. However, since we will be relating equivalence of symbols with equivalence of these operators, we will focus our attention solely on {\em bounded} symbols $\phi$. The symbol $\phi$ for $A^{\gamma}_{\phi}$ is not unique. Indeed  \cite{MR2363975},
\begin{equation}\label{Agammazerooooo}
A^{\gamma}_{\phi} \equiv 0 \iff \phi \in \overline{\gamma} \overline{H^{2}} + \gamma H^{2}.
\end{equation}
In a similar way \cite{MR3592981},  
\begin{equation}\label{Agammazerooooo1}
A^{\gamma, \alpha}_{\phi} \equiv 0 \iff \phi \in \overline{\gamma} \overline{H^{2}} + \alpha H^{2}.
\end{equation}

We now apply our main equivalence theorem (Theorem \ref{2.13A}) to these types of operators. To this end, let 
$\H_1 = \K_{\theta}$, $ \H_2 = \K_{\alpha}$, $\K_1 = \K_{\eta}$, $ \K_{2} = \K_{\gamma},$
where $\theta, \alpha, \eta$, and $\gamma$ are inner functions. Recall from Proposition \ref{2.2A} that for inner functions $\theta_1, \theta_2$ we have that $a \K_{\theta_2} = \K_{\theta_1}$ with $a \in L^{\infty}$ if and only if $a \in \mathscr{G}\!H^{\infty}$ and 
$
\theta_1 = h_{-} \theta_2 a$ and $h_{-} \in \mathscr{G}\!\overline{H^{\infty}}$.
From the identity 
$$1 = \theta_1 \bar{\theta}_1 = h_{-} \theta_2 a \bar{h}_{-} \bar{\theta}_{2} \bar{a} = h_{-} a \bar{h}_{-} \bar{a}$$
we see that $h_{-} \bar{a} = \bar{h}_{-}^{-1} a^{-1} = k \in \C$ with $|k| = 1$ and so $h_{-} = k \bar{a}^{-1}$ and we can assume without loss of generality that 
$$
\K_{\theta_1} = a \K_{\theta_2} \iff \theta_{1} = \bar{a}^{-1} \theta_2 a \quad \mbox{with $a \in \mathscr{G}\!H^{\infty}$}.
$$
With this in mind, a version of Theorem \ref{2.13A} for truncated Toeplitz operators becomes the following. 

\begin{Theorem}\label{6.1A}
Let $\theta, \alpha, \eta, \gamma$ be inner functions such that 
$$
\theta = \bar{a}_{1}^{-1} \eta a_1, \quad \alpha = \bar{a}_{2}^{-1} \gamma a_2 \quad \mbox{with $a_1, a_2 \in \mathscr{G}\!H^{\infty}$}.
$$
Then for any $\phi \in L^{\infty}$, $A_{\phi}^{\theta, \alpha}$ is equivalent to $A_{\widetilde{\phi}}^{\eta, \gamma}$ with $\widetilde{\phi} = \bar{a}_{2} \phi a_1$, and we have 
\begin{equation}\label{6.4}
A_{\phi}^{\theta, \alpha} = A_{\bar{a}_{2}^{-1}}^{\gamma, \alpha} A_{\widetilde{\phi}}^{\eta, \gamma} A_{a_{1}^{-1}}^{\theta, \eta} = A_{\bar{a}_{2}^{-1}}^{\gamma, \alpha} A_{\widetilde{\phi}}^{\eta, \gamma} M_{a_{1}^{-1}}|_{\K_{\theta}}.
\end{equation}
In particular, if $\theta = \alpha$, $\eta = \gamma$, and $a_1 = a_2 = a$, the identity in \eqref{6.4} becomes 
\begin{equation}\label{6.5}
A_{\phi}^{\alpha} = A_{\bar{a}^{-1}}^{\gamma, \alpha} A_{\widetilde{\phi}}^{\gamma} A_{a^{-1}}^{\alpha, \gamma} = A_{\bar{a}^{-1}}^{\gamma, \alpha} A_{\widetilde{\phi}}^{\gamma} M_{a^{-1}}|_{\K_{\alpha}},
\end{equation}
with $\widetilde{\phi} = \phi |a|^2$.
\end{Theorem}

We will have more to say about truncated Toeplitz operators in \S \ref{TTO6} and we will give another reformulation of Theorem \ref{2.13A} involving dual truncated Toeplitz operators in sections \S \ref{DTTO07}.

\section{Multipliers and generalized Toeplitz kernels and ranges}\label{sectionthree}

In this section, we focus on how the equivalence of generalized Toeplitz operators, induced by multipliers, relates to the kernels and ranges of these operators. 

Let us keep the notation from before: $\H, \H_1, \H_2, \K, \K_1, \K_2$ are closed subspaces of $L^2$ and $\phi, \psi \in L^{\infty}$. One of the main motivations for studying equivalence of operators of the form \eqref{9aiuerhgjfkdvc}, with symbols $\phi$ and $\psi$ related by a factorization $\phi = \bar{a}_2 \psi a_1$ as in Theorem \ref{2.13A}, stems from the study of Toeplitz operators and singular integral equations. In fact, the characterization of Toeplitz kernels and ranges, invertibility, and Fredholm properties is strongly related to the existence of an appropriate factorization of their symbols (for instance a Wiener--Hopf factorization or an almost periodic factorization). This allows us to establish an equivalence with Toeplitz operators with monomial symbols or simple exponential symbols -- which are much easier to understand.

Note that when $\phi \in L^{\infty}$ with $|\phi| = 1$ on $\T$, we have 
\begin{equation}\label{3.1}
\ker T_{\phi} \not = \{0\} \iff \phi = \bar{z} \bar{I} \bar{F} F^{-1},
\end{equation}
where $I$ is an inner function and $F \in H^2$ is outer \cite{MR3806717, MR2215727}. If $F \in \mathscr{G}\!H^{\infty}$, then for $\theta = z I$ we have that $\phi = \bar{F} \bar{\theta} F^{-1}$ and, by Theorem \ref{2.13A}, $T_{\phi}$ is equivalent to $T_{\bar{\theta}}$, with 
$$\ker T_{\phi} = F \ker T_{\bar{\theta}} = F \K_{\theta}.$$
This is true because if $g_{+} \in \mathscr{G}\!H^{\infty}$ and $g_{-} \in \mathscr{G}\!\overline{H^{\infty}}$, then for any $\widetilde{\phi} \in L^{\infty}$, 
\begin{equation}\label{3.2}
\ker T_{\widetilde{\phi} g_{+}} = g_{+}^{-1} \ker T_{\widetilde{\phi}} \quad \mbox{and} \quad \ker T_{\widetilde{\phi} g_{-}} = \ker T_{\widetilde{\phi}}.
\end{equation}
One may ask about the relationship between the kernels of two Toeplitz operators when the symbol  of one of them is multiplied by another function in $\mathscr{G}\!L^{\infty}$ which is not necessarily in $\mathscr{G}\!H^{\infty}$ or $\mathscr{G}\!\overline{H^{\infty}}$. If that function is a non-constant inner function, which also belongs to $H^{\infty} \setminus \mathscr{G}\!H^{\infty}$, we have that 
\begin{equation}\label{3.3}
\ker T_{\widetilde{\phi} \theta} \subsetneq \ker T_{\widetilde{\phi}} \quad \mbox{and} \quad \theta \ker T_{\widetilde{\phi} \theta} \subsetneq \ker T_{\widetilde{\phi}}.
\end{equation}

If we now consider Hankel operators instead, we see that with the same notation, if $g_{+} \in \mathscr{G}\!H^{\infty}$ and $\theta$ is inner,
\begin{equation}\label{3.4}
\ker H_{\widetilde{\phi} g_{+}} = \ker H_{\widetilde{\phi}},
\end{equation}
\begin{equation}\label{3.5}
\ker H_{\widetilde{\phi} \theta} \supseteq \ker H_{\widetilde{\phi}}   \quad \mbox{and} \quad \theta \ker H_{\widetilde{\phi} \theta} \subset \ker H_{\widetilde{\phi}},
\end{equation}
while no simple relation corresponds to the second equality in \eqref{3.2} for Hankel operators. 

In this section we show that by interpreting the relations \eqref{3.2} through \eqref{3.5} in terms of multipliers, we can put them in a general context and, along the way, get a better understanding of them. 

If two Hilbert space operators $A$ and $B$ are equivalent via \eqref{endnfddnfndffff}, with $A = E B F$, where $E$ and $F$ are invertible operators, it is clear that their kernels and ranges are isomorphic with 
$
\ker A = F^{-1} \ker B$ and $\operatorname{Ran} A = E \operatorname{Ran} B.$
More generally we have the following. 

\begin{Proposition}\label{3.1A}
Let $A \in \mathscr{B}(\K_1, \K_2)$ and $B \in \mathscr{B}(\H_1, \H_2)$  with $A = E B F$, where $E \in \mathscr{B}(\H_2, \K_2)$ and $F \in \mathscr{B}(\K_1, \H_1)$. 
\begin{enumerate}
\item If $E$ is invertible, then 
$F \ker A \subset \ker B$ and $ E \operatorname{Ran} B \supseteq \operatorname{Ran} A.$
\item If $F$ is invertible, then 
$\ker B \subset F \ker A$ and $E \operatorname{Ran} B \subset \operatorname{Ran} A.$
\end{enumerate}
\end{Proposition}

\begin{proof}
For the proof of (a), note that 
$$
A f = 0  \iff E B F f = 0 
 \iff B F f = 0
 \implies F f \in \ker B.
$$
Moreover, 
$$
A f  = g  \iff E B F f = g
 \iff B F f = E^{-1} g
 \iff E^{-1} g \in \operatorname{Ran} B.
$$
For the proof of (b), note that 
\begin{align*}
B f = 0 &  \iff B F (F^{-1} f) = 0\\
& \implies E B F(F^{-1} f) = 0\\
&  \iff A(F^{-1} f) = 0\\
 & \implies F^{-1} f \in \ker A.
\end{align*}
Moreover, 
$
B f  = g  \implies E B f = E g
 \iff E B F (F^{-1} f) = E g$
\end{proof}

Concerning the case where the operators $A$ and $B$ are generalized Toeplitz operators, one can take into account Theorem \ref{2.8A}, Remark \ref{2.13B}, and Proposition \ref{3.1A} to formulate the following.

\begin{Proposition}\label{3.2A}
Let $a_1, a_2 \in L^{\infty}$ with $a_1 \K_1 \subset \H_1$, $a_2 \K_2 \subset \H_2$, and let $\phi = \bar{a}_2 \psi a_1$. 
\begin{enumerate}
\item If $a_2 \in \mathscr{G}\!L^{\infty}$ and $a_2 \K_2 = \H_2$, then
$$
 a_1 \ker T_{\phi}^{\K_1, \K_2} \subset \ker T_{\psi}^{\H_1, \H_2}
\quad \mbox{and} \quad 
T_{\bar{a}_2}^{\H_2, \K_2} \operatorname{Ran}(T_{\psi}^{\H_1, \H_2}) \supseteq \operatorname{Ran}(T_{\phi}^{\K_1, \K_2}).
$$
\item If $a_1 \in \mathscr{G}\!L^{\infty}$ and $a_1 \K_1 = \H_1$, then
$$
\ker T_{\psi}^{\H_1, \H_2} \subset a_1 \ker_{\phi}^{\K_1, \K_2}
\quad \mbox{and} \quad 
T_{\bar{a}_2}^{\H_2, \K_2} \operatorname{Ran}(T_{\psi}^{\H_1, \H_2}) \subset \operatorname{Ran}(T_{\phi}^{\K_1, \K_2}).
$$
\end{enumerate}
\end{Proposition}

Some immediate consequences of Proposition \ref{3.2A} now follow. 

\begin{Corollary}\label{3.5A}
\begin{enumerate}
Let $a_1, a_2 \in L^{\infty}$. 
\item If $a_1 \K_1 \subset \H_1$ we have 
\begin{equation}\label{3.13}
a_1 \ker T_{\psi a_1}^{\K_1, \H_2} \subset \ker T_{\psi}^{\H_1, \H_2} \; \; \mbox{and} \; \; 
\operatorname{Ran}(T_{\psi}^{\H_1, \H_2}) \supseteq \operatorname{Ran}(T_{\psi a_1}^{\K_1, \H_2}).
\end{equation}
\item If $a_2 \K_2 \subset \H_2$ we have 
\begin{equation}\label{3.15}
\ker T_{\psi}^{\H_1, \H_2} \subset \ker T_{\bar{a}_2 \psi}^{\H_1, \K_2} \; \;  \mbox{and} \; \;  
T_{\bar{a}_2}^{\H_2, \K_2} \operatorname{Ran}(T_{\psi}^{\H_1, \H_2}) \subset \operatorname{Ran}(T_{\bar{a}_2 \psi}^{\H_1, \K_2}).
\end{equation}
\item If $a_1, a_2 \in \mathscr{G}\!L^{\infty}$ with $a_1 \K_1 = \H_1$, $a_2 \K_2 = \H_2$, and $\phi = \bar{a}_2 \psi a_1$, then
\begin{equation}\label{3.17}
a_1 \ker T_{\phi}^{\K_1, \K_2} = \ker T_{\psi}^{\H_1, \H_2} \;  \mbox{and} \; 
\operatorname{Ran}(T_{\phi}^{\K_1, \K_2}) = T_{\bar{a}_2}^{\H_2, \K_2} \operatorname{Ran}(T_{\psi}^{\H_1, \H_2}).
\end{equation}
\end{enumerate}
\end{Corollary}

\begin{Corollary}\label{3.7a}
If $a_1 \K_1 = \H_1$, then for any $\H_2$
\begin{equation}\label{3.10a}
a_{1}^{-1} \ker T_{\psi}^{\H_1, \H_2} = \ker T_{\psi a_{1}}^{\K_1, \H_2}.
\end{equation}
If $a_2 \K_2 = \H_2$, then for any $\H_1$, 
\begin{equation}\label{3.10b}
\ker T_{\psi}^{\H_1, \H_2} = \ker T_{\bar{a}_2 \psi}^{\H_1, \K_2}.
\end{equation}
\end{Corollary}

\begin{Remark}\label{3.6A}
\begin{enumerate}
\item For Toeplitz operators on $H^2$, the equalities from \eqref{3.2} can be seen as particular cases of \eqref{3.10a} and \eqref{3.10b}, while the inclusion in \eqref{3.3} can be see as a particular case of \eqref{3.15} with $a_2 = \theta$ and $\psi = \widetilde{\phi} \theta$. 

\item For Hankel operators, the equality \eqref{3.4} is a particular case of \eqref{3.10b}, while the first inclusion in \eqref{3.5} follows from \eqref{3.15} with $a_2 = \bar{\theta}$ and $\psi = \widetilde{\phi}$ and the second inclusion follows from \eqref{3.13}.


\item From the previous results we see that if $a_1$ is a multiplier from $\K_1$ into (onto) $\H_1$, then $a_1$ is also a multiplier from $\ker T_{\bar{a}_2 \psi a_1}^{\K_1, \K_2}$ into (onto) $\ker T_{\psi}^{\H_1, \H_2}$ if $a_2 \in \mathscr{G}\!L^{\infty}$, $a_2 \K_2 = \H_2$. In particular, $a_1$ is a multiplier from $\ker T_{\psi a_1}^{\K_1, \H_2}$ into (onto) $\ker  T_{\psi}^{\H_1, \H_2}$. 
\end{enumerate}
\end{Remark}

For Toeplitz operators on $H^2$, if the first equality in \eqref{3.17} holds for some $a_2 \in \mathscr{G}\!H^{\infty}$, then we must have $\phi = \bar{a}_2 \psi a_1$ for some $a_2 \in \mathscr{G}\!H^{\infty}$, i.e., 
$$\ker T_{\phi} = a_{1}^{-1} \ker T_{\psi} \implies \phi = \bar{a}_2 \psi a_1 \quad \mbox{for some $a_2 \in \mathscr{G}\!H^{\infty}$}$$
by \cite[Prop.~2.16]{MR3806717} (because in this case we have $\ker T_{ \phi} = \ker T_{a_1 \psi}$). It is therefore natural to ask if, for generalized Toeplitz operators, some sort of converse to Corollary \ref{3.5A}(c) is true. Namely, if \eqref{3.17} holds for some $a_1 \in \mathscr{G}\!L^{\infty}$ with $a_1 \K_1 = \H_1$, is it the case that the symbols $\phi$ and $\psi $ are related by $\phi = \bar{a}_2 \psi a_1$ for some $a_2 \in \mathscr{G}\!L^{\infty}$ with $a_2 \K_2 = \H_2$? This is not quite true. However, from Corollary \ref{3.7a}, we do have the following. 

\begin{Proposition}\label{3,8A}
Let $a_1 \in \mathscr{G}\!L^{\infty}$ with $a_1 \K_1 = \H_1$. If
$$\ker T_{\phi}^{\K_1, \K_2} = a_{1}^{-1} \ker T_{\psi}^{\H_1, \H_2},$$
then there exists a $\widetilde{\phi} \in L^{\infty}$ such that 
$$\ker T_{\phi}^{\K_1, \K_2} = \ker T_{\widetilde{\phi}}^{\K_1, \K_2}$$
and $\widetilde{\phi}  = \bar{a}_{2} \psi a_1$ for any $a_2$ such that $a_2 \H_2 = \K_2$.
\end{Proposition}

Note how this gives somewhat of a converse to Theorem \ref{2.13A}.

\section{Complex selfadjointness}\label{ALunitaryhhh}

For a closed subspace $\K$ of $L^2$, consider a mapping $C_{\K}: \K \to \K$ for which 
$C_{\K}(f + a g) = C_{\K} f + \bar{a} C_{\K} g$ for all $f, g, \in \K$ and $a \in \C$. Such a mapping is {\em antilinear}. For an antilinear mapping $C_{\K}$, one can define an antilinear adjoint $C_{\K}^{*}: \K \to \K$ which satisfies 
$\langle C_{\K} f, g\rangle = \overline{\langle f, C^{*}_{\K} g\rangle}$ for all $f, g \in \K$. The mapping $C_{\K}$ is called an {\em antilinear unitary mapping} if $C_{\K}^{*} = C_{\K}^{-1}$. We will give some examples in a moment.

\begin{Remark}\label{,,xmmx}
In many settings, one imposes the additional condition that $C_{\K}^2 = I_{\operatorname{d}}$ and such {\em involutive} and antilinear unitary mappings are called {\em conjugations}. Many of the $C_{\K}$ we will present below will be involutive. However, this extra hypothesis is not necessary for our results and so we leave it out. The version without the involutive assumption was explored in \cite{CSA} while the version with the involutive assumption was explored in \cite{MR3254868, MR2187654, MR2302518}. 
\end{Remark}

Let $C_{\K}$ denote an antilinear unitary mapping on $\K$. For $a \in \mathscr{G} L^{\infty}$ with $a \K = \H$,  this next result shows that, under certain conditions, the operators 
$$F = M_{a}|_{\K} \quad \mbox{and} \quad E = F^{*} = P_{\K} \bar{a} P_{\H}|_{\H}$$
from Proposition \ref{2.10AA}  induce a corresponding antilinear unitary mapping $C_{\H}$ on $\H$. 

\begin{Proposition}\label{4.2N}
Let $\H$ and $\K$ be closed subspaces of $L^2$ and $C_{\K}$ be an antilinear unitary mapping on $\K$, $F$ be an invertible operator from $\H$ onto $\K$, and $E = F^{*}$. Then for 
\begin{equation}\label{5554448uuUU}
C_{\H} = F^{-1} C_{\K} F,
\end{equation} 
the following are equivalent: 
\begin{enumerate}
\item $C_{\H}$ is an antilinear unitary mapping on $\H$; 
\item $C_{\K} (F E) = (F E) C_{\K}$.
\item $C_{\H} = E C_{\K} E^{-1}$. 
\end{enumerate}
\end{Proposition}

\begin{proof}
With $C_{\H}$ defined to be $F^{-1} C_{\K} F$, we see that $C_{\H}$ is antilinear and invertible with
$C_{\H}^{*} = F^{*} C^{*}_{\K} (F^{-1})^{*} = F^{*} C_{\K}^{-1} (F^{-1})^{*}$
and $C_{\H}^{-1} = F^{-1} C_{\K}^{-1} F$. Moreover, 
\begin{align*}
C_{\H}^{*} = C_{\H}^{-1} & \iff F^{*} C_{\K}^{-1} (F^{-1})^{*} = F^{-1} C_{\K}^{-1} F\\
& \iff F F^{*} C_{\K}^{-1} (F F^{*})^{-1} = C_{\K}^{-1}\\
& \iff  (F F^{*}) C_{\K} = C_{\K} (F F^{*})\\
& \iff (FE) C_{\K} = C_{\K} (FE)\\
& \iff E C_{\K} E^{-1} = F^{-1} C_{\K} F.  \qedhere
\end{align*}
\end{proof}

Recall from Remark \ref{,,xmmx} that our antilinear isometry $C_{\K}$ is called a conjugation if $C_{\K}^2 = I_{\operatorname{d}}$. Note that a conjugation satisfies $C = C^{-1}$. 

\begin{Proposition}
Let $C_{\H}$ be defined as in \eqref{5554448uuUU}. Then $C_{\H}^2 = I_{\operatorname{d}} \iff C_{\K}^2 = I_{\operatorname{d}}$. 
\end{Proposition}

An operator $T$ on $\H$ is {\em complex selfadjoint} with respect to an antilinear isometry $C$ on $\H$, if $C T C^{-1} = T^{*}$ \cite{CSA}.  If the antilinear unitary mapping also satisfies the involutive criterion $C^2 = I_{d}$, then $T$ is called a $C$-{\em symmetric operator}. These types of operators were explored in great detail in \cite{MR3254868, MR2187654, MR2302518}.

As we will see shortly, some generalized Toeplitz operators are complex self adjoint. The main theorem of this section relates the complex self adjointness of two equivalent generalized Toeplitz operators. 

\begin{Theorem}\label{comcpclcicic}
For closed subspaces $\H$ and $\K$ of $L^2$ and $\phi, \psi \in L^{\infty}$, suppose that $T_{\phi}^{\K}$ and $T_{\psi}^{\H}$ are equivalent via 
$$E = T_{\bar{a}}^{\H, \K} \quad \mbox{and} \quad F = T_{a}^{\K, \H},$$
where $a \in \mathscr{G} L^{\infty}$ with $a \K = \H$.  Also suppose that  $C_{\H}$ is an antilinear unitary mapping on $\H$ such that $F^{-1} C_{\H} F = E C_{\H} E^{-1}$. Then $T_{\psi}^{\H}$ is complex selfadjoint with respect to $C_{\H}$ if and only if $T_{\phi}^{\K}$ is complex selfadjoint with respect to the antilinear unitary mapping $C_{\K} = F^{-1} C_{\H} F$. 
\end{Theorem}

\begin{proof}
Note that 
\begin{align*}
C_{\K} T_{\phi}^{\K} C_{\K}^{-1}  = (T_{\phi}^{\K})^{*} & \iff (E C_{\H} E^{-1})(E T_{\psi}^{\H} F)(F^{-1} C_{\H}^{-1} F) = E (T_{\psi}^{\H})^{*} F\\
& \iff C_{\H} T_{\psi}^{\H} C_{\H}^{-1}  = (T_{\psi}^{\H})^{*} \qedhere
\end{align*}
\end{proof}

It is natural to also consider a conjugation on $\K$ which is equivalent, via the operators $E$ and $F$ defined in Theorem \ref{comcpclcicic}, to the antilinear unitary mapping $C_{\H}$ on $\H$.

\begin{Proposition}\label{P4.5aN}
With the same assumptions as in Proposition \ref{4.2N}, and with 
\begin{equation}\label{SCPI}
\widetilde{C}_{\K} = E C_{\H} F,
\end{equation}
the following are equivalent:
\begin{enumerate}
\item $\widetilde{C}_{\K}$ is an antilinear unitary mapping on $\K$;
\item $(FE) C_{\H} (FE) = C_{\H}$;
\item $E C_{\H} F = F^{-1} C_{\H} E^{-1}$.
\end{enumerate}
\end{Proposition}

\begin{Proposition}
Let $\widetilde{C}_{\K}$ be defined as in Proposition \ref{P4.5aN}  and assume that it is an anti-linear unitary mapping.
Then $C_{\H}^2 = I_{\operatorname{d}}$ if and only if $\widetilde{C}_{\K}^{2} = I_{\operatorname{d}}$.
\end{Proposition}

\begin{proof}
Notice that 
\begin{align*}
\widetilde{C}_{\K}^2 & = E C_{\H} (FE C_{\H}) F = E C_{\H}^2 E^{-1} F^{-1} F = I_{\operatorname{d}} \qedhere.
\end{align*}
\end{proof}

Note that, even if $\widetilde{C}_{\K}$, defined by Proposition \ref{SCPI}, is an antilinear unitary mapping, we are unable to conclude, in general, that $T_{\phi}^{\K}$ is complex selfadjoint with respect to $\widetilde{C}_{\K}$ when $T_{\psi}^{\H}$ is complex selfadjoint with respect to $C_{\H}$. Indeed, in general, if $F$ is not a unitary operator, we have 
\begin{align*}
\widetilde{C}_{\K} T_{\phi}^{\K} \widetilde{C}^{-1}_{\K} & = F^{-1} C_{\H} E^{-1} (E T_{\psi}^{\H} F) F^{-1} C_{\H}^{-1} E^{-1}\\
& = F^{-1} (T_{\psi}^{\H})^{*} E^{-1}\\
& \not = F^{*} (T_{\psi}^{\H})^{*} E^{*} = (T_{\phi}^{\K})^{*}.
\end{align*}
When $F$ is a unitary operator, we have the following.

\begin{Corollary}\label{4.5cN}
Let $C_{\K}$ and $\widetilde{C}_{\K}$ be defined as in Propositions \ref{4.2N} and \ref{P4.5aN}, respectively. If $F$ is unitary then 
\begin{enumerate}
\item $C_{\K}$ is an antilinear unitary mapping.
\item $C_{\K} = \widetilde{C}_{\K}$,
\item $T_{\phi}^{\K} = E T_{\psi}^{\H} F$ is complex selfadjoint with respect to $C_{\K}$ if and only if $T_{\psi}^{\H}$ is complex selfadjoint with respect to $C_{\H}$.
\end{enumerate}
\end{Corollary}

\begin{Example}
For an inner function $\gamma$, the truncated Toeplitz operator $A_{\phi}^{\gamma}$ (defined in \eqref{ATTO999}) is complex selfadjoint \cite{MR2187654} (see also \cite[p.~291]{MR3526203}) with respect to the conjugation $C_{\gamma}$ on the model space $\K_{\gamma}$ defined by 
$$(C_{\gamma} f)(\xi) = \gamma(\xi) \overline{\xi f(\xi)}.$$

Now let $\alpha$ be any inner function and $a \in \mathscr{G}\!H^{\infty}$ such that $a \K_{\alpha} = \K_{\gamma}$, i.e., $a = \bar{\alpha} \gamma \bar{h}$ for some $g \in \mathscr{G}\!H^{\infty}$. Since $\alpha, \gamma$ are inner, we have 
$|a| = |h|$ and that $a \bar{a} = h \bar h$ if and only if $a h^{-1} = \bar{h} \bar{a}^{-1}$. Since the left hand side of the last equality belongs to $H^{\infty}$ while the right hand side belongs to $\overline{H^{\infty}}$, we conclude that $a h^{-1} = \bar{h} \bar{a}^{-1}$ is a constant and therefore $h = c a$. We can assume that $c = 1$ and it follows that 
\begin{equation}\label{IIN}
\gamma = \bar{a}^{-1} \alpha a.
\end{equation}
Therefore, 
 Theorem \ref{6.1A} says that 
$$A_{|a|^2 \psi}^{\alpha} = A_{\bar{a}}^{\gamma, \alpha} A_{\psi}^{\gamma} A_{a}^{\alpha, \gamma}.$$
Theorem \ref{comcpclcicic} now says that $A_{|a|^2 \psi}^{\alpha}$ is complex symmetric with respect to the conjugation $\widetilde{C}$ defined on $\K_{\alpha}$ by $\widetilde{C} f = a^{-1} C_{\gamma} a f$. Note that $\widetilde{C}$ satisfies the equivalent conditions of Proposition \ref{4.2N}. Indeed, using \ref{IIN}, we see that for any $f \in \K_{\alpha}$
\begin{align*}
a^{-1} C_{\gamma} a f & = a^{-1} \gamma \bar{z} \bar{a} \bar{f}\\
& = a^{-1} (\bar{a}^{-1} \alpha a) \bar{z} \bar{a} \bar{f}\\
& = \alpha \bar{z} \bar{f}
\end{align*}
and 
\begin{align*}
A_{\bar{a}}^{\gamma, \alpha} C_{\gamma} A_{\bar{a}^{-1}}^{\alpha, \gamma} f & = P_{\alpha} \bar{a} P_{\gamma} C_{\gamma}P_{\gamma} \bar{a}^{-1} P_{\alpha} f\\
& = P_{\alpha} \bar{a} P_{\gamma} C_{\gamma}\bar{a}^{-1} f\\
& = P_{\alpha} \bar{a} P_{\gamma}\gamma \bar{z} a^{-1} \bar{f}\\
& = P_{\alpha} \bar{a} P_{\gamma} \bar{a}^{-1} \alpha \bar{z} \bar{f}\\
& = P_{\alpha} \bar{a} \bar{a}^{-1} \alpha \bar{z} \bar{f}\\
& = \alpha \bar{z} \bar{f}
\end{align*}
and we verify that $\widetilde{C}$ turns out to be the usual conjugation on $\K_{\alpha}$.
\end{Example}

\section{Kernels of Toeplitz operators}\label{sectionfive}

When $\H_1 = \H_2 = \K_1 = \K_2 = H^2$, then $a H^2 = H^2$ if and only if $a \in \mathscr{G}\!H^{\infty}$. We say that $\phi, \psi \in L^{\infty}$ are {\em equivalent}, and write $\phi \sim \psi$, if and only if there are $a_{+} \in \mathscr{G}\!H^{\infty}$ and $a_{-} \in \mathscr{G}\!\overline{H^{\infty}}$ such that
\begin{equation}\label{5.1}
\phi = a_{-} \psi a_{+}.
\end{equation}
In this case, the Toeplitz operators $T_{\phi}$ and $T_{\psi}$ are equivalent (see Theorem \ref{2.19A}). Conversely, if $T_{\phi}$ and $T_{\psi}$ are equivalent via \eqref{2.9}, then by Corollary \ref{3.5A} there exists an $a_{+} \in \mathscr{G}\!H^{\infty}$ such that $\ker T_{\psi} = a_{+} \ker T_{\psi}$. It follows that $\psi$ and $\phi$ satisfy a relation of the form \eqref{5.1}, so that $\phi \sim \psi$ (see the remarks before Proposition \ref{3,8A}).

It is often the case that the function  $\psi$ in \eqref{5.1} is the conjugate of an inner function $\theta$, i.e., $\phi \in \mathscr{G}\!L^{\infty}$ and 
\begin{equation}\label{5.2}
\phi = a_{-} \bar{\theta} a_{+}.
\end{equation}
This is the case when  the function $\phi$ is an invertible H\"{o}lder continuous function on $\T$ with exponent $\mu \in (0, 1)$ which has a negative index with respect to the origin  \cite[Ch.~3, Cor.~5,2]{MR881386} and we take $\theta = z^n$. 
It is also the case for certain classes of almost periodic functions $\phi$ on $\R$, where the inner function $\theta$ is an exponential of the form $\theta(x) = e^{i a x}$, $a \in (0, \infty)$ \cite{MR1898405}. 

Let us give a few examples of equivalent functions. 

\begin{Example}[Finite Blaschke products]\label{FBPEx}

Suppose $a_1, \ldots, a_n$  are points in the open unit disk $\D$ (repetitions allowed) and 
$$\alpha(z) = \prod_{j = 1}^{n} \frac{z - a_j}{1 - \bar{a}_j z}$$ is a  finite Blaschke product \cite{MR3793610}.  Note that $\alpha$ is an inner function. For  any $z \in \T$ we have the identity 
\begin{equation}\label{5.3}
\alpha(z) = \frac{1}{z^n}  \prod_{j = 1}^{n} (z - a_j) \cdot z^n \cdot \prod_{j = 1}^{n} \frac{1}{1 - \bar{a}_j z} = \alpha_{-}(z) \cdot z^n \cdot \alpha_{+}(z),
\end{equation}
where $\alpha_{+}, \bar{\alpha}_{-} \in \mathscr{G}\!H^{\infty}$, and hence $\alpha \sim z^n$. As an immediate consequence, we see that all finite Blaschke products of the same degree (same number of zeros -- counting multiplicity) are equivalent. 
\end{Example}

\begin{Example}[Generalized Frostman shifts]\label{5.2A}
Let $h \in H^{\infty}$ with $\|h\|_{\infty} < 1$. For any inner function $\theta$ we define 
\begin{equation}\label{A1.1}
\theta_{\bar{h}} := \frac{\theta - \overline{h}}{1 - h \theta}.
\end{equation}
Using the fact that $\theta$ is inner, and thus $\theta \bar{\theta} = 1$ almost everywhere on $\T$, it follows that 
$
\theta_{\bar{h}} \in L^{\infty}$ and $|\theta_{\bar{h}}| = 1$ almost everywhere on $\T$. 
Furthermore, a calculation shows that 
$$
\theta_{\bar{h}} = a_{-} \theta a_+,
$$
where 
$$
a_{-} = 1 - \bar{h} \bar{\theta} \in \mathscr{G}\!\overline{H^{\infty}} \quad \mbox{and} \quad a_+  = (1 - h \theta)^{-1} \in \mathscr{G}\!H^{\infty}.
$$
Thus $\theta \sim \theta_{\bar{h}}$. Observe that $\theta$ is inner while $\theta_{\bar{h}}$ is unimodular on $\T$, but not necessarily inner (since it may not be the boundary values of a bounded analytic function on $\D$).
An interesting case arises when $\bar{h}$ is a constant function equal to $a \in \D$. In this case, the function $\theta_{\bar{h}}$ from \eqref{A1.1} becomes 
\begin{equation}\label{Frostyusdfsfgsd}
\theta_{a}(z) = \frac{\theta(z) - a}{1 - \bar{a} \theta(z)},
\end{equation}
which is a {\em Frostman shift} of $\theta$ (and also an inner function).  Thus $\theta$ and $\theta_a$ are equivalent inner functions for every $a \in \D$.
A classical result of Frostman \cite{Frostman} (see also \cite[p.~75]{MR2261424}) says that $\theta_a$ is a Blaschke product for ``most'' $a \in \D$. 
\end{Example}

It follows from \eqref{5.2} that $a_{+}^{-1}$ is a multiplier from $\K_{\theta}$ onto $\ker T_{\phi}$, thus yielding two isomorphisms between $\ker T_{\phi}$ and $\K_{\theta}$ as in Proposition \ref{2.10AA}. This gives us the following (taking Proposition \ref{2.5A} and Corollary \ref{3.5A} into account). 

\begin{Proposition}\label{5.3A}
If $\phi = a_{-} \bar{\theta} a_{+}$, where $ a_{+}, \bar{a}_{-} \in \mathscr{G}\!\overline{H^{\infty}}$, then 
\begin{enumerate}
\item $\ker T_{\phi} = a_{+}^{-1} \K_{\theta}$.
\item $(\ker T_{\phi})^{\perp} = \bar{a}_{+} \K_{\theta}^{\perp}$.
\end{enumerate}
\end{Proposition}

Moreover, if \eqref{5.2} holds, when studying $\ker T_{\phi}$ we can assume, without loss of generality, that $|\phi| = 1$ almost everywhere on $\T$ because $\ker T_{\phi} = \ker T_{\widetilde{\phi}}$ with $\widetilde{\phi} = \bar{a}_{+} \bar{\theta} a_{+}$ (since $a_{-}, \bar{a}_{+} \in \mathscr{G}\!\overline{H^{\infty}}$ \cite{MR3806717}). On the other hand, if $|\phi| = 1$ almost everywhere on $\T$ and $\phi = a_{-} \bar{\theta} a_{+}$, it is not difficult to see that $a_{-}$ and $\bar{a}_{+}^{-1}$ must differ by a multiplicative constant. Thus, if $|\phi| = 1$ almost everywhere on $\T$ and \eqref{5.2} holds, then 
$$(\ker T_{\phi})^{\perp} = a_{-}^{-1} \K_{\theta}^{\perp}.$$

The equality in Proposition \ref{5.3A} has a connection with Hayashi's representation of Toeplitz kernels. In \cite{MR853630}, Hayashi proved that the kernel $K$ of a Toeplitz operator  with symbol $\phi$ can be written as $K = f_{+} \K_{\alpha}$, where $f_{+}$ is outer, $\alpha$ is inner with $\alpha(0) = 0$, and $f_{+}$ multiplies the model space $\K_{\alpha}$ isometrically onto $K$. When $f_{+} \in \mathscr{G}\!H^{\infty}$, we must have $\phi = f_{-} \bar{\alpha} f_{+}^{-1}$ for some $f_{-} \in \mathscr{G}\!\overline{H^{\infty}}$  (see \S \ref{sectionthree}).
In general, neither the outer function $f_{+}$ nor the inner function $\alpha$ in Hayashi's representation of a Toeplitz kernel are explicitly known. However,  the representation in Proposition \ref{5.3A} is explicit whenever one can factor $\phi$ as in \eqref{5.2}. This can be done for a wide class of functions and provides an isomorphism between $\ker T_{\phi}$ and $\K_{\theta}$ through an invertible bounded multiplier. From here, one might ask when is this isomorphism isometric, as in Hayashi's representation. The answer lies in the relationship between the two isomorphisms from $\ker T_{\phi}$ onto $\K_{\theta}$, namely
\begin{equation}\label{5.9}
I_{1} := M_{a_{+}}|_{\ker T_{\phi}} \quad \mbox{and} \quad I_{2} := P_{\theta} \bar{a}_{+}^{-1} I_{\operatorname{d}}|_{\ker T_{\phi}} = (I_{1}^{-1})^{*},
\end{equation}
and, by Corollary \ref{2.18A}, can be given in terms of a truncated Toeplitz operator as follows. 

\begin{Proposition}\label{5.7N}
Let $\phi = a_{-} \bar{\theta} a_{+}$, where $a_{+}, \bar{a}_{-} \in \mathscr{G}\!H^{\infty}$ and $\theta$ is inner. Then 
$M_{a_{+}^{-1}}: \K_{\theta} \to \ker T_{\phi}$ is an isometric isomorphism if and only if 
\begin{equation}\label{8.8N}
A^{\theta}_{1 - |a_{+}|^{-2}} = 0
\end{equation}
or, equivalently, 
\begin{equation}\label{5.9N}
1 - |a_{+}|^{-2} \in \overline{\theta H^2} + \theta H^2.
\end{equation}
\end{Proposition}

\begin{Example}[Generalized Crofoot transform]\label{GCTT}
For $h \in H^{\infty}$ with $\|h\|_{\infty} < 1$ and $\theta$ inner, recall
$$\theta_{\bar{h}} = \frac{\theta - \bar{h}}{1 - h \theta},$$  the generalized Frostman shift from Example \ref{5.2A}. For any $k \in \C \setminus \{0\}$,
$$\overline{\theta_{\bar{h}}} = a_{-} \bar{\theta} a_{+},$$
where  
$$a_{+} := k^{-1} (1 - h \theta) \in \mathscr{G}\!H^{\infty} \quad \mbox{and} \quad a_{-} = k (1 - \bar{h} \bar{\theta})^{-1} \in \mathscr{G}\!\overline{H^{\infty}}.$$ From Proposition \ref{5.3A} we know that 
$$J_{\bar{h}} := \frac{k}{1 - h \theta}$$ is a multiplier from $\K_{\theta}$ onto $\ker T_{\overline{\theta_{\bar{h}}}}$. We call the operator of multiplication by $J_{\bar{h}}$ a {\em generalization of the  Crofoot transform}, for it includes, as a special case, the Crofoot transform as defined in \cite{MR2363975} -- as we will now show. 

From Proposition \ref{5.7N}, $M_{J_{\bar{h}}}$ will be an isometric isomorphism from $\K_{\theta}$ onto $\ker T_{\overline{\theta_{\bar{h}}}}$ if and only if 
\begin{equation}\label{5.14N}
1 - \frac{|k|^2}{|1 - h \theta|^2} \in \overline{\theta H^2} + \theta H^2.
\end{equation}
Observe that 
\begin{align*}
1 - \frac{|k|^2}{|1 - h \theta|^2} & = \frac{(1 - h \theta)(1 - \bar{h} \bar{\theta}) - |k|^2}{(1 - h \theta)(1 - \bar{h} \bar{\theta})}\\
& = \frac{h (\bar{h} - \theta) + \bar{h} (h - \bar{\theta}) + 1 - h \bar{h} - |k|^2}{(1 - h \theta)(1 - \bar{h} \bar{\theta})}\\
& = - \frac{h}{1 - h \theta} \theta - \frac{\bar{h}}{1 - \bar{h} \bar{\theta}} \bar{\theta} + \frac{1- |h|^2 - |k|^2}{(1 - h \theta)(1 - \bar{h} \bar{\theta})}
\end{align*}
and that 
$$\frac{h}{1 - h \theta}  \in H^{\infty} \subset H^2 \quad \mbox{and} \quad \frac{\bar{h}}{1 - \bar{h} \bar{\theta}}  \in \overline{H^{\infty}} \subset \overline{H^2}.$$
Therefore, \eqref{5.14N} is satisfied if and only if 
\begin{equation}\label{5.15N}
\frac{1 - |h|^2 - |k|^2}{(1 - h \theta)(1 - \bar{h} \bar{\theta})} \in \overline{\theta H^2} + \theta H^2.
\end{equation}
This last condition is certainly satisfied when $h = \bar{p} \in \D$ is a constant function, and then choosing the constant $k$ such that $|k|^2 = 1 - |h|^2$. In this case, $\theta_{\bar{h}} = \theta_{p}$ is an inner function (a Frostman shift of $\theta$ from \eqref{Frostyusdfsfgsd}), $\ker T_{\overline{\theta_{\bar{h}}}} = \K_{\theta_p}$ is a model space, and $M_{J_{\bar{h}}}  = M_{J_p}$ is the {\em Crofoot transform} which maps $\K_{\theta}$ isometrically onto $K_{\theta_p}$ (considered in \cite{MR2363975}). Whether \eqref{5.15N} can be satisfied for some $k \in \C \setminus \{0\}$ when $h$ is not a constant function is an open question worthy of further investigation. 
\end{Example}

\section{Truncated Toeplitz operators}\label{TTO6}

In this section, we continue (and refine)  our discussion of truncated Toeplitz operators begun in \S \ref{section2}.
From Theorem \ref{6.1A}, with $\theta = \alpha, \eta = \gamma$, and $\phi \in L^{\infty}$, and from Corollary \ref{2.18A}, we have the following.

\begin{Theorem}\label{6.1N}
If $\alpha$ and $\gamma$ are inner functions, $\phi \in L^{\infty}$, and $\alpha = \bar{a}^{-1} \gamma a$, with $a \in \mathscr{G}\!H^{\infty}$, then $\K_{\alpha} = a \K_{\gamma}$ and 
\begin{equation}\label{6.2N}
A_{\phi}^{\alpha} = A_{\bar{a}^{-1}}^{\gamma, \alpha} A_{\widetilde{\phi}}^{\gamma} A_{a^{-1}}^{\alpha, \gamma} = A_{\bar{a}^{-1}}^{\gamma, \alpha} A_{\widetilde{\phi}}^{\gamma} M_{a^{-1}}|_{\K_{\alpha}},
\end{equation}
where $\widetilde{\phi} = \phi |a|^2$. The relation in \eqref{6.2N} defines a unitary equivalence between $A_{\phi}^{\alpha}$ and $A_{\phi |a|^2}^{\gamma}$ if and only if 
$$1 - |a|^2 \in \overline{\gamma H^2} + \gamma H^2,$$
or, equivalently, 
$$1 - |a|^{-2} \in \overline{\alpha H^2} + \alpha H^2,$$ and, in this case, 
\begin{equation}\label{6.3N}
A_{\phi}^{\alpha} = a A_{\widetilde{\phi}}^{\gamma} a^{-1} I_{\operatorname{d}}|_{\K_{\alpha}}.
\end{equation}
\end{Theorem}

An example where \eqref{6.3N} holds is given by the Crofoot transform acting on a truncated Toeplitz operator, as shown in \cite{MR2363975} (see also Example \ref{GCTT}).

Note that if we consider an operator on $\K_{\alpha}$ by replacing $A_{\bar{a}_{2}^{-1}}^{\gamma, \alpha}$ in \eqref{6.4} by the inverse of $A_{a_{1}^{-1}}^{\theta, \eta}$ (which is equal to multiplication by $a_1$ on $\K_{\theta}$), we do not, in general, obtain a truncated Toeplitz operator. This can be seen with the following example. 

\begin{Example}\label{6.3}
For any $\lambda \in \overline{\D}$, and any finite Blaschke product $\theta$, define 
\begin{equation}\label{6.9}
k_{\lambda}^{\theta}(z) = \frac{1 - \overline{\theta(\lambda)} \theta(z)}{1 - \bar{\lambda} z} \quad \mbox{and} \quad 
\widetilde{k}_{\lambda}^{\theta}(z) = \frac{\theta(z) - \theta(\lambda)}{z - \lambda}.
\end{equation}
Now let $\alpha$ be the simple degree two Blaschke product defined by 
$$\alpha(z) = \frac{(z - \tfrac{1}{2})(z - \tfrac{1}{3})}{(1 - \tfrac{1}{2} z)(1 - \tfrac{1}{3} z)},$$
which can be factored as $\alpha = \bar{a}^{-1} z^2 a,$
where 
$$a(z) = \frac{1}{(1 - \tfrac{1}{2} z)(1 - \tfrac{1}{3} z)} \in \mathscr{G}\!H^{\infty},$$
and let $\phi(z) = \alpha(z)/z$. Then defining 
$$A := A_{a^{-1}}^{\alpha, z^2} A_{\alpha/z}^{\alpha} A_{a}^{z^2, \alpha} = a^{-1} A_{\alpha/z}^{\alpha} a I_{\operatorname{d}}|_{\K_{z^2}},$$ 
which is an operator from $\K_{z^2}$ to itself, we see that for any $f \in \K_{z^2} = \operatorname{span}\{1, z\}$, 
\begin{equation}\label{6.44N}
A f = a(0) f(0) a^{-1} \widetilde{k}_{0}^{\alpha} = a(0) a^{-1} \cdot \widetilde{k}_{0}^{\alpha} \otimes k_{0}^{z^2} (f).
\end{equation}
Thus, $A$ is a rank-one operator. If $A$ were a truncated Toeplitz operator, it would have be a scalar multiple of one of the following types: 
$k_{\lambda}^{z^2} \otimes  \widetilde{k}_{\lambda}^{z^2}$, $\widetilde{k}_{\lambda}^{z^2} \otimes k_{\lambda}^{z^2}$ for $\lambda \in \D$ or  $ k_{\zeta}^{z^2} \otimes k_{\zeta}^{z^2}$, where $\zeta \in \T$ \cite{MR2363975}. Comparing this with the form of $A$ in \eqref{6.44N}, we see that $a^{-1} \widetilde{k}_{0}^{\alpha}$ should be a constant multiple of $\widetilde{k}_{0}^{z^2}$, which it is not. 
\end{Example}

Theorem \ref{6.1N} allows us to reduce the study of truncated Toeplitz operators on arbitrary model spaces of dimension $n$ to that of a truncated Toeplitz operator on $\K_{z^n}$, as in the following example. 

\begin{Example}\label{6.5NN}
Let $\alpha$ be the degree two Blaschke product from Example \ref{6.3} and consider the question of invertibility of the operator $A_{\phi}^{\alpha}$, where 
$$\phi(z) = \frac{(z - \tfrac{1}{2})(z - \tfrac{1}{3}) (z^2 + 1)}{z^2}.$$
Since $\alpha = \bar{a}^{-1} z^2 a$ with $a \in \mathscr{G}\!H^{\infty}$ as in Example \ref{6.3}, we can use Theorem \ref{6.1N} to conclude that $A_{\phi}^{ \alpha}$ is equivalent to $A_{\widetilde{\phi}}^{z^2}$ with 
$$\widetilde{\phi}(z) = 1 + \tfrac{5}{6} z + z^2 h_{+}$$ for some $h_{+} \in H^{\infty}$. Moreover
$$A_{\widetilde{\phi}}^{z^2}  = A_{1 + \frac{5}{6} z}^{z^2}$$ which is invertible \cite[Thm.~4.1]{MR3398735}. Thus, $A_{\phi}^{\alpha}$ is invertible. 
\end{Example}

Also note that \eqref{6.5} provides an equivalence between the inverses of the two operators $A_{\phi}^{\alpha}$ and $A_{\widetilde{\phi}}^{\gamma}$, if they exist. We have the following result which can be obtained from \cite[Thm.~2.4]{MR3398735}. 

\begin{Proposition}\label{6.5A}
Let $\theta$ be an inner function and let $\phi \in L^{\infty}$ be such that 
$$G = \begin{bmatrix}
\bar{\theta} & 0\\
\phi & \theta
\end{bmatrix}$$
admits a Wiener--Hopf factorization of the form $G = G_{-} G_{+}$, with $\overline{G}_{-}^{\,\pm}, G_{+}^{\pm} \in (H^2)_{2 \times 2}$ \cite{MR3398735, MR881386}. Moreover, denote 
$G_{+}^{-1} = [g_{ij}^{+}]$ and $G_{-}^{-1} =  [g_{ij}^{-}]$, $ i, j = 1, 2.$
Then, for all $f \in \K_{\theta}$, 
$$(A_{\phi}^{\theta})^{-1} f = P_{\theta}(g_{11}^{+} P_{+} g_{12}^{-} +  g_{12}^{+} P_{+} g_{22}^{-}) f.$$
\end{Proposition}

When $\theta = z^n$, necessary and sufficient conditions for $A_{\phi}^{\theta}$ to be invertible were obtained in \cite{MR3398735} and the factorization of $G$ itself can be similarly determined in explicit form by solving the relatively simple Riemann--Hilbert problem $G G_{+}^{-1} = G_{-}$ ($G$ is a $2 \times 2$ triangular matrix of polynomials with diagonal entries $\bar{z}^n$ and $z^n$). Proposition \ref{6.5A} provides an explicit and simple expression for the inverse of $A_{\widetilde{\phi}}^{z^n}$ and from \eqref{6.5} we can then obtain an expression for the inverse of $A_{\phi}^{\alpha}$ (which does not depend on the basis that one chooses for $\K_{\alpha}$).


As we have seen in Example \ref{FBPEx},  for two finite Blaschke products $\alpha$ and $\gamma$ of equal degree, there is an $a \in \mathscr{G} H^{\infty}$ such that $a \K_{\gamma} = \K_{\alpha}$. This allows us to prove the following. 

\begin{Proposition}\label{Prandnksd}
Let $\alpha$ and $\gamma$ be two finite Blaschke products of equal degree. Then two truncated Toeplitz operators $A^{\alpha}_{\phi}$ and $A^{\gamma}_{\psi}$ are equivalent if and only if they have the same rank. 
\end{Proposition}

\begin{proof}
This result will follow from the following linear algebra result. Two $n \times n$ matrices $A$ and $B$ are equivalent (i.e., there are $n \times n$ invertible matrices $E$ and $F$ with $A = E B F$) if and only if they have the same rank. Indeed, if $A$ is equivalent to $B$ then clearly $A$ and $B$ have the same rank. Conversely, if $A$ and $B$ have the same rank $r \leq n$ then by the singular value decomposition there are $n \times n$ unitary matrices $U, V, W, Q$ such that 
$A = U \Sigma_{A} V$ and $B = W \Sigma_{B} Q,$
where $\Sigma_{A}, \Sigma_{B}$ are the $n \times n$ diagonal matrices 
$$\Sigma_{A} = \operatorname{diag}(\sigma_1, \sigma_2, \ldots, \sigma_r, 0, 0, \ldots 0),$$
$$\Sigma_{B} = \operatorname{diag}(\lambda_1, \lambda_2, \ldots, \lambda_r, 0, 0, \ldots, 0).$$
Of course, $\sigma_j$ and $\lambda_j$ are nonzero for all $1 \leq j \leq r$.
If 
$$D = \operatorname{diag}\Big(\frac{\sigma_1}{\lambda_1}, \frac{\sigma_2}{\lambda_2}, \ldots, \frac{\sigma_r}{\lambda_r}, 1, 1, \ldots, 1\Big),$$
then $D$ is invertible and 
\begin{align*}
A  &= U \Sigma_{A} V\\
 & = U D \Sigma_{B} V\\
& = U D W^{*} B Q^{*} V\\
 &= (U D W^{*}) B (Q^{*} V)\\
 &= E B F
\end{align*} and so $A$ is equivalent to $B$. 
\end{proof}
 
The proof of Proposition \ref{Prandnksd} says that {\em any} matrix is equivalent to a Toeplitz matrix of equal rank. In fact, one can say more. A result from \cite[Thm.~6.2]{MR2648079} says that any square matrix is similar to some  truncated Toeplitz operator on some finite dimensional model space. One can only go so far with these type of results since small matrices are similar to Toeplitz matrices while large ones are generally not \cite{MR1839449}.

We end with an application to  the essential spectrum of two equivalent truncated Toeplitz operators. For a bounded operator $T$ on a Hilbert space $\H$, recall that the {\em essential spectrum} of $T$, denoted by $\sigma_{e}(T)$, is the set of all $\lambda \in \C$ such that the operator $T - \lambda I$ is not Fredholm. 

For an inner function $\alpha$, let 
$$\Sigma(\alpha) := \Big\{\xi \in \T: \varliminf_{z \to \xi} |\alpha(z)| = 0\Big\}$$ denote the {\em boundary spectrum} of the inner function $\alpha$. One can show that $\Sigma(\alpha)$ consists of  the accumulation points of the zeros of $\alpha$ along with the support of the associated singular measure for $\alpha$ in its Riesz factorization \cite[p.~152]{MR3526203}. Another known result \cite[p.~204]{MR3526203} is that 
\begin{equation}\label{77yHHGGTTT555}
\Sigma(\alpha) = \sigma_{e}(A^{\alpha}_{z}).
\end{equation}

Let  $\mathscr{R}$ denote the set of rational functions whose poles are not in $\T$. Also define  
$$\mathscr{R}_{+} := P_{+} \mathscr{R} = \mathscr{R} \cap H^{\infty} \quad \mbox{and} \quad \mathscr{R}_{-} := P_{-} \mathscr{R} =  \mathscr{R} \cap \overline{z H^{\infty}}.$$
As used before, $P_{+}$ is the Riesz projection from $L^2$ onto $H^2$ and $P_{-} = I - P_{+}$.
 If $\gamma$ is inner and also satisfies  $\gamma = \bar{a}^{-1} \alpha a$ for some $a \in \mathscr{G}\!H^{\infty}$, then by Theorem \ref{6.1A}, the operators 
$A^{\alpha}_{R - \lambda}$ and $A^{\gamma}_{|a|^2 (R - \lambda)}$ are equivalent for all $\lambda \in \C$. 
Thus,  $A^{\alpha}_{R - \lambda}$ is Fredholm if and only if the operator $A^{\gamma}_{|a|^2 (R - \lambda)}$ is Fredholm. Using the notation $P_{\gamma H^2}$ for the orthogonal projection of 
$L^2$
onto $\gamma H^2$, we see that 
\begin{align*}
A^{\gamma}_{|a|^2 (R - \lambda)} & = 
P_{\gamma} |a|^2(R - \lambda) P_{\gamma}|_{\K_{\gamma}}\\
& = \Big(P_{\gamma} (R - \lambda) P_{\gamma} \bar{a} P_{\alpha} a P_{\gamma} + P_{\gamma} (R - \lambda) P_{-} (\bar{a} P_{\alpha} a P_{\gamma})\\
& \qquad  + P_{\gamma} (R - \lambda) P_{\gamma H^2} (\bar{a} P_{\alpha} a P_{\gamma})\Big)|_{\K_{\gamma}}\\
& = A_{R - \lambda}^{\gamma} A_{\bar{a}}^{\alpha, \gamma} A_{a}^{\gamma, \alpha} + P_{\gamma} (R - \lambda) P_{-} (\bar{a} P_{\alpha} a P_{\gamma})|_{\K_{\gamma}}\\
& \qquad  + P_{\gamma} (R - \lambda) P_{\gamma H^2} (\bar{\alpha} P_{\alpha} \alpha P_{\gamma})|_{\K_{\gamma}},
\end{align*}
where the last two operators above are finite rank (see Lemma \ref{finiteranksksksk} below). Since the operators $A_{\bar{a}}^{\alpha, \gamma}$ and $A_{a}^{\gamma, \alpha}$ are invertible, we obtain the following.

\begin{Theorem}\label{rationalsymbolskks}
With the assumptions above,
$\sigma_{e}(A^{\alpha}_{R}) = \sigma_{e}(A^{\gamma}_{R})$.
\end{Theorem}

As mentioned above, the needed ingredient to complete the proof of Theorem \ref{rationalsymbolskks} is the following lemma.

\begin{Lemma}\label{finiteranksksksk}
If $R \in \mathscr{R}$, then $P_{\gamma} R P_{-}$ and $P_{\gamma} R P_{\gamma H^2}$ are finite rank operators. 
\end{Lemma}

\begin{proof}
We have $R = R_{+} + R_{-}$ with $R_{\pm} \in \mathscr{R}_{\pm}$. Then 
$$P_{\gamma} R P_{-} = P_{\gamma} (R_{-} + R_{+}) P_{-} = P_{\gamma} (P_{+} R_{+} P_{-})$$
and 
\begin{align*}
P_{\gamma} R P_{\gamma H^2}  &= P_{\gamma} (R_{-} + R_{+}) \gamma P_{+} \bar{\gamma} P_{+}\\
 &= P_{\gamma} R \gamma P_{+} \bar{\gamma} P_{+}\\
 &= P_{\gamma} \gamma (P_{-} R_{-} P_{+}) \bar{\gamma} P_{+}.
\end{align*}
By Kronecker's theorem on finite rank Hankel operators\ \cite{MR1630646}, it follows that $P_{\gamma} R_{+} P_{-}$ and $P_{\gamma} R_{-} P_{\gamma H^2}$ are finite rank operators. 
\end{proof}

Theorem \ref{rationalsymbolskks} along with \eqref{77yHHGGTTT555} gives another proof of an observation of Crofoot \cite[Thm.~14]{MR1313454}:
A necessary condition for $\alpha$ to be equivalent to $\gamma$ is that $\Sigma(\alpha) = \Sigma(\gamma)$.

\section{Dual truncated Toeplitz operators}\label{DTTO07}

This section formulates a version of Theorem \ref{2.13A} when the generalized Toeplitz operator becomes a dual truncated Toeplitz operator. If $\alpha$ is inner, recall from our discussion at the end of \S \ref{section2} that $P_{\alpha}$ denotes the orthogonal projection of $L^2$ onto the model space $\K_{\alpha}$. Then $Q_{\alpha} := I - P_{\alpha}$ is the orthogonal projection of $L^2$ onto 
$$\K_{\alpha}^{\perp} = \alpha H^{2} \oplus (H^{2})^{\perp}.$$ For $\phi \in L^{\infty}$, define 
$$D_{\phi}^{\gamma, \alpha}: \K_{\gamma}^{\perp} \to \K_{\alpha}^{\perp}, \quad D_{\phi}^{\gamma, \alpha}f := Q_{\alpha} (\phi f).$$ When $\gamma = \alpha$ we write $D^{\gamma, \gamma}_{\phi} = D_{\phi}^{\gamma}$. The operator $D_{\phi}^{\gamma}$ is a {\em dual truncated Toeplitz operator} on $\K_{\gamma}$  with symbol $\phi$, while $D_{\phi}^{\gamma, \alpha}$ is an {\em asymmetric dual truncated Toeplitz operator} from $\K_{\gamma}$ to $\K_{\alpha}$ with symbol $\phi$. These operators have received considerable attention recently  \cite{MCKB, MR3759573, BHal, Mbbhbbh}.  Unlike the criterion in \eqref{Agammazerooooo} and \eqref{Agammazerooooo1} for an asymmetric truncated Toeplitz operator to be the zero operator, we have the following (see \cite[Prop.~2.1]{MR3759573})
\begin{equation}\label{DTTOaerooo}
D^{\gamma, \alpha}_{\phi} \equiv 0 \iff \phi = 0.
\end{equation}
Thus, the symbol of a dual truncated Toeplitz operator is unique.

 Some results concerning the equivalence of dual truncated Toeplitz operators are known  when $\phi(z) = z$ and thus $D_{z}^{\alpha}$ is the dual of the compressed shift \cite{CamaraRoss}. For example, $D_{z}^{\alpha}$ is unitarily equivalent to $D_{z}^{\gamma}$ if and only if $|\alpha(0)| = |\gamma(0)|$. When $\alpha(0) \not = 0$ and $\gamma(0) \not = 0$, then $D_{z}^{\alpha}$ is similar to $D_{z}^{\gamma}$ and they are both similar to the bilateral shift on $L^2$ (i.e., the operator $f(\xi) \mapsto \xi f(\xi)$). When $\alpha(0) = 0$, $D_{z}^{\alpha}$ is unitarily equivalent to $S \oplus S^{*}$ on $H^2 \oplus H^2$, where $S$ is the unilateral shift. From here, one can characterize the invariant subspaces of $D_{z}^{\alpha}$ \cite{DanT}.

A version of Theorem \ref{2.13A} for dual truncated Toeplitz operators is the following. 

\begin{Theorem}\label{sevenpoiertbn77y}
Let $\theta$, $\eta, \alpha, \gamma$ be inner functions and $a_1, a_2 \in \mathscr{G} \overline{H^{\infty}}$ with 
$\theta = a_1 \eta \bar{a}_{1}^{-1}$ and $\alpha = a_{2} \gamma \bar{a}_{2}^{-1}$.
Then for any $\phi \in L^{\infty}$, the operator 
$D_{\phi}^{\theta, \alpha}$ is equivalent to $D_{\widetilde{\phi}}^{\eta, \gamma}$, where 
$\widetilde{\phi} = \bar{a}_2 \phi a_1$.  Moreover, 
 $D_{\bar{a}_{2}^{-1}}^{\gamma, \alpha}$ and $D_{a_{1}^{-1}}^{\theta, \eta}$ are bounded invertible operators such that 
$$D_{\phi}^{\theta, \alpha} = D_{\bar{a}_{2}^{-1}}^{\gamma, \alpha} D_{\widetilde{\phi}}^{\eta, \gamma} D_{a_{1}^{-1}}^{\theta, \eta}.$$
\end{Theorem}

\begin{Corollary}\label{sofhjgdfger4t5yhgfd7}
For inner functions $\gamma$, $\alpha$ and $a \in \mathscr{G} \overline{H^{\infty}}$ with 
$\alpha = a \gamma \bar{a}^{-1}$,
 let $\phi \in L^{\infty}$ and $\widetilde{\phi} = |a|^2 \phi$. Then 
$D^{\alpha}_{\phi} = (D^{\alpha, \gamma}_{a^{-1}})^{*} D^{\gamma}_{\widetilde{\phi}} D_{a^{-1}}^{\alpha, \gamma}.$
\end{Corollary}

One can wonder whether the equivalence in the previous result can be an isometric isomorphism, that is to say, $D_{a}^{\gamma, \alpha} = D_{\overline{a}^{-1}}^{\gamma, \alpha}$? Unlike the case with truncated Toeplitz operators discussed earlier (Theorem \ref{6.1N}), similarity in this case can only happen in a trivial way. Indeed, $D_{a}^{\gamma, \alpha} = D_{\overline{a}^{-1}}^{\gamma, \alpha}$ implies that 
$D_{a -\overline{a}^{-1}}^{\gamma, \alpha} \equiv 0$ and so, by \eqref{DTTOaerooo}, $a = \overline{a}^{-1}$. Since $H^{\infty} \cap \overline{H^{\infty}} = \C$, this would imply that $a = \overline{a}^{-1}$ and is a unimodular constant function. 

It follows from Corollary \ref{sofhjgdfger4t5yhgfd7} that 
\begin{equation}\label{7.8}
\ker D^{\theta}_{\phi}  = \bar{a} \ker D_{\widetilde{\phi}}^{\eta}. 
\end{equation}
As an application of this result we study the kernel of a dual truncated Toeplitz operator whose symbol $\phi$ does not belong to $\mathscr{G}\!L^{\infty}$ and thus whose study cannot be reduced to that of a truncated Toeplitz operator with symbol $\phi^{-1}$ \cite[\S 6]{MCKB}. 

We have that 
\begin{equation}\label{7.9}
\ker D^{\theta}_{\phi} = \{f_1 \in \K_{\theta}^{\perp}: \phi f_1 = f_{2+} \in \K_{\theta}\}.
\end{equation}
Noting that $f_1 \in \K_{\theta}^{\perp}$ if and only if $f_{1}  = f_{1-} + \theta f_{1+}$, where $f_{1-} \in (H^2)^{\perp}$ and $f_{1+} \in H^2$; and $f_{2+} \in \K_{\theta}$ if and only if $f_{2+} \in H^2$ and $\bar{\theta} f_{2+} = f_{2-} \in (H^2)^{\perp}$, the equation defining $\ker D_{\phi}^{\theta}$ is equivalent to the matricial Riemann--Hilbert problem 
\begin{equation}\label{7.10}
\begin{bmatrix} \bar{\theta} & 0\\ -1 & \phi \theta \end{bmatrix} \begin{bmatrix} f_{2+}\\ f_{1+}\end{bmatrix} + \begin{bmatrix} -1 & 0\\ 0 & \phi \end{bmatrix} \begin{bmatrix} f_{2-}\\ f_{1-}\end{bmatrix} = \begin{bmatrix} 0 \\ 0
\end{bmatrix}.
\end{equation}

Now let $\theta$ be a finite Blaschke product of degree $n$ and, for any inner function $\alpha$,  define $\phi(z) = \alpha(z) (z - 1)$. From Example \ref{FBPEx} we see that 
$$\theta = \overline{B_{+}^{-1}} z^n B_{+} \quad \mbox{with $B_{+} \in \mathscr{G}\!H^{\infty}$}.$$
Therefore, by Corollary \ref{sofhjgdfger4t5yhgfd7}, 
\begin{equation}\label{7.11}
D_{\phi}^{\theta} = D_{B_{+}}^{z^n, \theta} D_{\widetilde{\phi}}^{z^n} D_{\overline{B_{+}}}^{\theta, z^n}, \quad \mbox{where} \quad  \widetilde{\phi} = \overline{B_{+}^{-1}} \phi B_{+}^{-1},
\end{equation}
and, by \eqref{7.8}, 
\begin{equation}\label{7.12}
\ker D_{\phi}^{\theta} = D^{z^n, \theta}_{\overline{B_{+}^{-1}}} \ker D_{\widetilde{\phi}}^{z^n} = \overline{B_{+}^{-1}} \ker D_{\widetilde{\phi}}^{z^n}.
\end{equation}
Thus we have reduced the problem of determining $\ker D^{\theta}_{\phi}$ to that of determining $ \ker D_{\widetilde{\phi}}^{z^n}$. For this problem, we use \eqref{7.9} with $\theta$ and $\phi$ replaced by $z^n$ and $\widetilde{\phi}$ respectively to obtain 
\begin{equation}\label{7.13}
\bar{z}^n f_{2+} = f_{2-}
\end{equation}
and 
\begin{equation}\label{7.13a}
f_{2+} - B_{+}^{-1} \alpha (z - 1) \overline{B_{+}^{-1}} z^n f_{1+} = B_{+}^{-1} \alpha (z - 1) \overline{B_{+}^{-1}} f_{1-}.
\end{equation}
From  \eqref{7.13} we have that $f_{2+} = p_{n - 1}$, where $p_{n - 1} \in \mathscr{P}_{n - 1}$ (the polynomials of degree at most $n - 1$). Substituting this into the equation in \eqref{7.13a} we see that 
$$p_{n - 1} - B_{+}^{-1} \alpha (z - 1) \overline{B_{+}^{-1}} z^n f_{1+} = B_{+}^{-1} \alpha (z - 1) \overline{B_{+}^{-1}} f_{1-}$$ if and only if 
$$B_{+}^{-1} (z - 1) f_{1+} = \bar{\alpha} \bar{z}^{n} p_{n - 1} \overline{B_{+}} - \bar{\theta} \overline{B_{+}^{-1}} (z - 1) f_{1-}.$$
Moreover, since the left hand side of the last equality belongs to $H^2$ while the right hand side belongs to $\overline{H^2}$, both sides are equal to a constant. Due to the factor of $z - 1$ on the left hand side, this constant must be zero. This means that $f_{1+} = 0$ and
\begin{equation}\label{7.14}
\bar{\alpha} \bar{z}^n p_{n - 1} \overline{B_{+} }= \bar{\theta} \overline{B_{+}^{-1}} (z - 1) f_{1-}.
\end{equation}
It follows from \eqref{7.14} that 
$$\frac{p_{n - 1}}{z - 1} = \bar{\theta} \alpha z^n \overline{B_{+}^{-2}} f_{1-} \in L^2$$
and so $p_{n - 1}(1) = 0$. Therefore, if $n = 1$ we have that $p_{n - 1} \equiv 0$ and thus $\ker D_{\phi}^{\theta} = \{0\}$. 
When $n > 1$, we have 
\begin{equation}\label{7.15}
\bar{\alpha} \bar{z}^n \bar{B}_{+}^{2} p_{n - 2} = \bar{\theta} f_{1-} \quad \mbox{with $p_{n - 2} \in \mathscr{P}_{n - 2}$}.
\end{equation}
Let $\gamma = \operatorname{gcd}(\theta, z \alpha)$ (which is necessarily a finite Blaschke product). Then, from \eqref{7.15}, 
$$\frac{z^2 \alpha}{\gamma} B_{+}^{2} z^{n - 2} \overline{p_{n - 2}} = \frac{\theta}{\gamma} \overline{f_{1-}}.$$
From the fact that $z^{n - 2} \overline{p_{n - 2}} = \widetilde{p}_{n - 2} \in \mathscr{P}_{n - 2}$, we see that 
$$\frac{z \alpha}{\gamma} B_{+}^{2} \widetilde{p}_{n - 2} = \frac{\theta}{\gamma} \bar{z} \overline{f_{1-}} \quad \mbox{with $\bar{z} \overline{f_{1-}} \in H^2$}.$$ and thus, 
$$
\widetilde{p}_{n - 2} \in \frac{\theta}{\gamma} H^2.
$$
Let $k = \operatorname{deg}(\theta/\gamma)$. When $n - 2 < k$, we have $\widetilde{p}_{n - 2} = 0$ and so $\ker D_{\phi}^{\theta} = \{0\}$. When $n - 2 \geq k$ we have 
$\widetilde{p}_{n - 2} = p_{\theta/\gamma} \widetilde{p}_{n - 2 - k},$
where $p_{\theta/\gamma}$ is the numerator of $\theta/\gamma$ and $\widetilde{p}_{n - 2 - k} \in \mathscr{P}_{n - 2 - k}$. Hence, 
$$f_{1-} = \theta \overline{p_{\theta/\gamma}} \bar{\alpha} \bar{z}^{2} \overline{B_{+}^{2}} \overline{\widetilde{p}_{n - 2 - k}}.$$ Conversely, if $f_{1} = f_{1-}$ (note that $f_{1+} = 0$), then $f_{1} \in \ker D_{\widetilde{\phi}}^{z^n}$ and so $\operatorname{dim} \ker D_{\phi}^{\theta} = n - 1 - k$.

Summarizing these results we obtain the following.

\begin{Proposition}
Let $\theta$ be a Blaschke product of degree $n$ and, for any inner function $\alpha$, let $\phi(z) = \alpha(z) (z - 1).$ Furthermore, let 
$$\gamma = \operatorname{gcd}(\theta, z \alpha), \quad k = \operatorname{deg} \frac{\theta}{\gamma},$$
and 
$$d(z) = \prod_{j = 1}^{k} (1 - \overline{w_{j}} z),$$
where $w_1, \ldots, w_k$ are the zeros of $\theta/\gamma$ (counting multiplicity). Then
\begin{enumerate}
\item if $n \leq k + 1$, $\ker D_{\phi}^{\theta} = \{0\}$. 
\item if $n > k + 1$, 
$$\ker D_{\phi}^{\theta} = a \big(\overline{d} \cdot \overline{\big(\frac{\alpha z}{\gamma}\big)} \cdot \bar{z} \cdot \overline{\mathscr{P}_{n - k - 2}}) \subset (H^2)^{\perp}$$
and  $\operatorname{dim} \ker D_{\phi}^{\theta} = n - 1 - k$.
\end{enumerate}
\end{Proposition}

\section{Statements and declarations} 

\subsection*{Conflict of interest} On behalf of all of the authors, there is no conflict of interest. 

\subsection*{Data availability} No datasets were generated or analyzed during the current study.

\bibliographystyle{plain}

\bibliography{references}

\end{document}